\documentclass[12pt]{amsart}
\usepackage{amsfonts}
\usepackage{amssymb, amscd}
\usepackage{amsmath}
\input xy
\xyoption {all}

\usepackage{geometry}
\usepackage{amsthm}
\usepackage{amssymb}
\usepackage{latexsym}
\usepackage[dvips]{graphics}

\newtheorem{theorem}{Theorem}[section] 
\newtheorem{cor}[theorem]{Corollary}
\newtheorem{definition}[theorem]{Definition}
\newtheorem{ass}[theorem]{Assumptions}
\newtheorem{example}[theorem]{Example}
\newtheorem{lemma}[theorem]{Lemma}
\newtheorem{prop}[theorem]{Proposition}
\newtheorem{remark}[theorem]{Remark}
\newtheorem*{ack}{Acknowledgements}

\newcommand{\mh}{\mbox{MHM}}

\newcommand{\Ku}{\mbox{Kue}}
\newcommand{\ms}{\mbox{mHs}}
\newcommand{\ve}{\mbox{vect}}

\newcommand{\LL}{{\mathcal L}}

\newcommand{\OO}{{\mathcal O}}

\newcommand{\VV}{{\mathcal V}}

\newcommand{\FC}{{\mathcal F}}
\newcommand{\MC}{{\mathcal M}}
\newcommand{\PC}{{\mathcal P}}

\newcommand{\Z}{\mathbb{Z}}
\newcommand{\N}{\mathbb{N}}
\newcommand{\Q}{\mathbb{Q}}

\newcommand{\C}{\mathbb{C}}
\newcommand{\K}{\mathbb{K}}

\begin{document}

\title[Twisted genera of symmetric products]{Twisted genera of symmetric products}

\author[L. Maxim ]{Laurentiu Maxim}
\address{L.  Maxim: Department of Mathematics,  University of Wisconsin-Madison, 480 Lincoln Drive, Madison, WI 53706, USA.}
\email {maxim@math.wisc.edu}

\author[J. Sch\"urmann ]{J\"org Sch\"urmann}
\address{J.  Sch\"urmann : Mathematische Institut,
          Universit\"at M\"unster,
          Einsteinstr. 62, 48149 M\"unster,
          Germany.}
\email {jschuerm@math.uni-muenster.de}



\thanks{The first author partially supported by NSF-1005338. The second author partially supported by the SFB 878 ``groups, geometry and actions".}
\date{\today}

\begin{abstract} We give a new proof of formulae for the generating series of (Hodge) genera
of symmetric products $X^{(n)}$ 
 with coefficients, which hold for complex quasi-projective varieties $X$ with any kind of singularities, and which include many of the classical results in the literature as special cases. Important specializations of our results include generating series for extensions of
Hodge numbers and Hirzebruch's $\chi_y$-genus to the singular setting and, in particular, generating series for Intersection cohomology Hodge numbers and Goresky-MacPherson Intersection cohomology signatures of symmetric products
of complex projective varieties. Our proof applies to more general situatons
and is based on equivariant K\"{u}nneth formulae and pre-lambda structures on the coefficient theory of a point, $\bar{K}_0(A(pt))$, with $A(pt)$ a Karoubian $\Q$-linear tensor category. Moreover, Atiyah's approach to power operations in $K$-theory also applies in this context
to  $\bar{K}_0(A(pt))$, giving  a nice description of the important related Adams
operations. This last approach also allows us to introduce  
 very interesting coefficients on the symmetric products $X^{(n)}$.
\end{abstract}

\maketitle

\tableofcontents


\section{Introduction}

Some of the most interesting examples of global orbifolds are the {\it symmetric products}  $X^{(n)}$, $n \geq 0$,  of a smooth complex algebraic variety $X$. The $n$-fold symmetric product of $X$ is defined as $X^{(n)}:=X^n/{\Sigma_n}$, i.e.,  the quotient of the product of $n$ copies of $X$ by the natural action of the symmetric group  on $n$ elements, $\Sigma_n$. If $X$ is a smooth projective curve,  symmetric products are of fundamental importance in the study of the Jacobian variety of $X$ and other aspects of its geometry, e.g., see \cite{Mac1}. If $X$ is a smooth algebraic surface, $X^{(n)}$ is used to understand the topology of the $n$-th Hilbert scheme $X^{[n]}$ parametrizing  closed zero-dimensional subschemes of length $n$ of $X$, e.g., see \cite{Che, GS, GLM0}, and also \cite{Che, GLM1} for higher-dimensional generalizations.  For the purpose of this note we shall assume that $X$ is a (possibly singular) complex quasi-projective variety, therefore its symmetric products are algebraic varieties as well.\\

A {\it generating series} for a given invariant $\mathcal{I}(-)$ of symmetric products of complex algebraic varieties  is an expression of the form 
$$\sum_{n \geq 0} \mathcal{I}( X^{(n)} ) \cdot t^n,$$ 
provided  $ \mathcal{I}( X^{(n)} ) $ is defined for all $n$. The aim is to calculate such an expression solely in terms of invariants of $X$. Then the corresponding invariant of the $n$-th symmetric product $X^{(n)}$ is equal to the coefficient of $t^n$ in the resulting expression in invariants of $X$.\\

For example, there is a well-known formula due to Macdonald \cite{Mac} for the generating series of the {\em Betti numbers} $b_k(X):=dim(H^k(X,\Q))$,
Poincar\'{e} polynomial $P(X):= \sum_{k \geq 0}  b_k(X)\cdot (-z)^k$, and topological {\em Euler characteristic}
$\chi(X):=\sum_{k \geq 0} (-1)^k\cdot b_k(X)=P(X)(1)$ of a compact triangulated space $X$:
\begin{equation} \label{Macd-betti} \begin{split}
\sum_{n \geq 0} \left(\:\sum_{k \geq 0}  b_k(X^{(n)})\cdot (-z)^k \right) \cdot t^n &=
\prod_{k \geq 0} \left(\frac{1}{1-z^kt}\right)^{(-1)^k\cdot b_k(X)} \\
& = \exp \left( \sum_{r \geq 1} P(X)(z^r)  \cdot \frac{t^r}{r} \right) 
\end{split}
\end{equation}
and, respectively,
\begin{equation} \label{Macd-euler}
\sum_{n \geq 0}  \chi(X^{(n)})  \cdot t^n = (1-t)^{-\chi(X)} = 
\exp \left( \sum_{r \geq 1}  \chi(X) \cdot \frac{t^r}{r} \right) \:.
\end{equation}
For the last equalities in (\ref{Macd-betti}) and (\ref{Macd-euler}), recall that $-\log(1-t)=\sum_{r \geq 1} \frac{t^r}{r}$.\\

A Chern class version of formula (\ref{Macd-euler}) was obtained by Ohmoto in \cite{Oh} for the Chern-MacPherson classes of \cite{MP}. Moonen \cite{M} obtained generating series for the {\em arithmetic genus} $\chi_a(X):=\sum_{k \geq 0} (-1)^k\cdot dim H^k(X,\OO_X)$
of symmetric products of a complex projective variety:
\begin{equation} \label{Moon-euler}
\sum_{n \geq 0}  \chi_a(X^{(n)})  \cdot t^n = (1-t)^{-\chi_a(X)} = 
\exp \left( \sum_{r \geq 1}  \chi_a(X) \cdot \frac{t^r}{r} \right) 
\end{equation}
and, more generally, for the Baum-Fulton-MacPherson homology Todd classes (cf. \cite{BFM}) of symmetric products of  projective varieties. In \cite{Za}, Zagier obtained such generating series for the {\em signature} $\sigma(X)$ and, resp.,  $L$-classes of symmetric products of compact triangulated  (rational homology) manifolds, e.g., for a  complex projective homology manifold $X$ of  pure even complex dimension he showed that
\begin{equation}\label{Za-sig}
\sum_{n \geq 0} {\sigma}(X^{(n)}) \cdot t^n=\frac{(1+t)^{\frac{\sigma(X)-\chi(X)}{2}}}{(1-t)^{\frac{\sigma(X)+\chi(X)}{2}}}.
\end{equation}

Borisov and Libgober obtained in \cite{BL3} generating series for the Hirzebruch $\chi_y$-genus  and, more generally, for elliptic genera of symmetric products of smooth compact varieties (cf. also \cite{Zh} for other calculations involving the Hirzebruch $\chi_y$-genus). Generating series for the
{\em Hodge numbers} $h_{(c)}^{p,q,k}(X):=h^{p,q}(H_{(c)}^k(X,\Q))$ and, resp., the {\em $E$-polynomial}
$$e_{(c)}(X):=\sum_{p,q \geq 0} e_{(c)}^{p,q}(X)\cdot y^px^q, \quad \text{with} \quad e_{(c)}^{p,q}(X):=
\sum_{k \geq 0} (-1)^k\cdot h_{(c)}^{p,q,k}(X),$$ 
of the cohomology (with compact support) of a quasi-projective variety $X$ 
(endowed with Deligne's mixed Hodge structure \cite{De}) are considered by Cheah in \cite{Che}: 
\begin{equation} \label{Che-betti}
\sum_{n \geq 0} \left(\:\sum_{p,q,k \geq 0} h_{(c)}^{p,q,k}(X^{(n)})\cdot y^px^q(-z)^k \right) \cdot t^n =
\prod_{p,q,k \geq 0} \left(\frac{1}{1-y^px^qz^kt}\right)^{(-1)^k\cdot h_{(c)}^{p,q,k}(X)} 
\end{equation}
and, respectively,
\begin{equation} \label{Che-e} 
\sum_{n \geq 0} e_{(c)}(X^{(n)})  \cdot t^n = \prod_{p,q\geq 0} 
\left(\frac{1}{1-y^px^qt}\right)^{ e_{(c)}^{p,q}(X)} 
= \exp \left( \sum_{r \geq 1}  e_{(c)}(X)(y^r,x^r) \cdot \frac{t^r}{r} \right) \:.
 \end{equation}
Note that for $X$ a quasi-projective variety one has 
$$b_k(X)=\sum_{p,q} h^{p,q,k}(X) \quad \text{ and} \quad \chi(X)=e_{(c)}(X)(1,1)\:,$$ 
so that one gets back (\ref{Macd-betti}) and (\ref{Macd-euler}) above 
by substituting $(y,x)=(1,1)$ in (\ref{Che-betti}) and (\ref{Che-e}), respectively. Similarly, using the relation 
$$\chi_{-y}(X)=e(X)(y,1)=:\sum_{p\geq 0} f^p\cdot y^p$$ 
for a projective manifold $X$, one gets in this case for the {\em Hirzebruch $\chi_y$-genus} the identities:
\begin{equation} \label{chi-y-smooth} 
\sum_{n \geq 0} \chi_{-y}(X^{(n)})  \cdot t^n = \prod_{p\geq 0} 
\left(\frac{1}{1-y^pt}\right)^{f^p(X)} 
= \exp \left( \sum_{r \geq 1}  \chi_{-y^r}(X) \cdot \frac{t^r}{r} \right) \:.
 \end{equation}
Note that for $X$ a complex projective manifold one also has 
$$\chi_a(X)=\chi_0(X)=e(X)(0,1) \quad \text{and} \quad \sigma(X)=\chi_1(X)=e(X)(-1,1)\:,$$ so that one gets back (\ref{Moon-euler}) and
(\ref{Za-sig}) for this case by letting $(y,x)=(0,1)$ and $(y,x)=(-1,1)$, respectively.\\

Lastly, Ganter \cite{Ga} gave conceptual interpretations of generating series formulae in homotopy theoretic settings, and Getzler formulated generating series results already in the general context of suitable ``K\"{u}nneth functors" with values in a ``pseudo-abelian $\Q$-linear tensor category" (which in \cite{Ge1} is called a ``Karoubian rring"), see  \cite{Ge1}[Prop.(5.4)].\\

The purpose of this note is to prove a very general generating series formula for such genera ``with coefficients", which holds for complex quasi-projective varieties with any kind of singularities, and which includes many of the above mentioned results as special cases. Important specializations of our result include, among others, generating series for extensions of 
{\em Hodge numbers and Hirzebruch's $\chi_y$-genus} to the singular setting with coefficients
$$\MC \in D^b\mh(X),$$
a complex of Saito's (algebraic) mixed Hodge modules (\cite{Sa,Sa0,Sa1}), and, in particular, generating series for {\em intersection cohomology Hodge numbers} and Goresky-MacPherson {\em intersection cohomology signatures}
 (\cite{GM1}) of complex projective varieties. A more direct proof of these results in the context of complexes of mixed Hodge modules has been recently given 
in \cite{MSS}. Here we supply a new (more abstract) proof of these results, relying on the theory of pre-lambda rings (e.g., see \cite{Kn}), which has the merit that it also applies to more general situations, as we shall explain later on.

Note that mixed Hodge module coefficients are already used in \cite{Ge2}, but with other techniques and applications in mind.
Besides the use of very general coefficients, our approach is close to \cite{Ge1}.
The use of coefficients not only gives more general results, but is also needed
for a functorial
characteristic class version of some of our results in terms of the homology Hirzebruch classes of Brasselet-Sch\"urmann-Yokura \cite{BSY}, which 
are treated in our recent paper \cite{CMSSY}.
This corresponding characteristic class version from \cite{CMSSY} unifies the mentioned results of
\cite{Oh, M, Za} for Chern-, Todd- and $L$-classes.\\

In this paper, the functorial viewpoint is only used in three important special cases:
\begin{enumerate}
\item[(a)] By taking the exterior product $\MC^{\boxtimes n} $ we get an 
object on the cartesian product $X^n$, which by \cite{MSS} is equivariant with respect to a corresponding  permutation action of the symmetric group $\Sigma_n$ on $X^n$.
\item[(b)] We push down the exterior product $\MC^{\boxtimes n} $  by the projection
$p_n : X^n\to X^{(n)}$ onto the $n$-th symmetric product $X^{(n)}$ to get a $\Sigma_n$-equivariant
object on $X^{(n)}$, from which one can take by \cite{BS} the $\Sigma_n$-invariant sub-object 
\begin{equation} \label{sym-proj}
\MC^{(n)}:=\PC^{sym}({p_n}_*\MC^{\boxtimes{n}})= ({p_n}_*\MC^{\boxtimes{n}})^{\Sigma_n} 
\in D^b\mh(X^{(n)})
\end{equation}
defined by the projector 
$\PC^{sym}= \frac{1}{n!}\sum_{\sigma \in \Sigma_n} \psi_{\sigma} =: (-)^{\Sigma_n}$, where the isomorphisms $\psi_{\sigma}: {p_n}_*\MC^{\boxtimes{n}}\to {p_n}_*\MC^{\boxtimes{n}}$
for $\sigma \in \Sigma_n$ are given by the action of $\Sigma_n$ on ${p_n}_*\MC^{\boxtimes{n}}$ (compare with the appendix).
Here we use the fact that $\Sigma_n$ acts trivially on $X^{(n)}$.
\item[(c)] We push down by the constant map $k$ to a point space $pt$ for calculating the invariants we are interested in:
$$k_*\MC^{(n)}, k_!\MC^{(n)}\in D^b\mh(pt) \simeq D^b(\ms^p) \:,$$
identifying by Saito's theory the abelian category $\mh(pt)$ of mixed Hodge modules on a point with the  abelian category $\ms^p$ of {\em graded polarizable $\Q$-mixed Hodge structures},
with $H^*(pt,k_{*(!)}\MC^{(n)})= H^*_{(c)}(X^{(n)},\MC^{(n)})$.
\end{enumerate}

And these are related by a  K\"{u}nneth formula (compare with Remark \ref{rem-kue}, and also with \cite{MSS}[Thm.1])
\begin{equation} \label{Kue-mHs} 
H^*_{(c)}(X^{(n)},\MC^{(n)})
\simeq (H^*_{(c)}(X^n,\MC^{\boxtimes n}))^{\Sigma_n} \simeq 
((H^*_{(c)}(X,\MC))^{\otimes{n}})^{\Sigma_n} 
\end{equation}
for the cohomology (with compact support). These are isomorphisms of graded groups of mixed Hodge structures. 

For suitable choices of our coefficients $\MC$, the ``symmetric power" $\MC^{(n)}$ becomes a highly interesting object on the symmetric product $X^{(n)}$. In this paper we are only interested in genera of $\MC^{(n)}$, i.e., invariants defined in terms of the mixed Hodge structure on 
$H_{(c)}^*(X^{(n)};\MC^{(n)})$. But for the characteristic class version from \cite{CMSSY} it is important to work directly on the singular space $X^{(n)}$.\\

The isomorphism (\ref{Kue-mHs}) of graded groups of mixed Hodge structures is the key ingredient for proving the following result (compare also with \cite{MSS}[Cor.2]):
\begin{theorem}\label{main1} Let $X$ be a complex quasi-projective variety and $\MC \in D^b\mh(X)$ a bounded complex of mixed Hodge modules on $X$. For $p,q,k\in \Z$, denote by
$$h_{(c)}^{p,q,k}(X,\MC):=h^{p,q}(H_{(c)}^k(X,\MC)):=
dim_{\C}(Gr^p_FGr^W_{p+q}H_{(c)}^k(X,\MC))$$  
the Hodge numbers, with generating polynomial 
$$h_{(c)}(X,\MC):=\sum_{p,q,k} h_{(c)}^{p,q,k}(X,\MC)\cdot y^px^q(-z)^k  \in \Z[y^{\pm 1},x^{\pm 1},z^{\pm 1}]\:.$$
Note that for $z=1$ this yields the $E$-polynomial
$$e_{(c)}(X,\MC):=\sum_{p,q} e_{(c)}^{p,q}(X,\MC)\cdot y^px^q \in \Z[y^{\pm 1},x^{\pm 1}] $$ 
 of the cohomology (with compact support) $H_{(c)}^*(X,\MC)$ of $\MC$,
with $$e_{(c)}^{p,q}(X,\MC):=\sum_{k \geq 0} (-1)^k\cdot h_{(c)}^{p,q,k}(X,\MC)\:.$$
For 
$$\MC^{(n)}:=({p_n}_*\MC^{\boxtimes{n}})^{\Sigma_n} \in D^b\mh(X^{(n)})$$ the $n$-th symmetric product of $\MC$,
the following generating series formulae hold:
\begin{equation} \label{gen-MHM1} 
\begin{split}
 \sum_{n \geq 0} \left(\:\sum_{p,q,k} h_{(c)}^{p,q,k}(X^{(n)},\MC^{(n)})\cdot y^px^q(-z)^k \right) \cdot t^n  
 &=\prod_{p,q,k} \left(\frac{1}{1-y^px^qz^kt}\right)^{(-1)^k\cdot h_{(c)}^{p,q,k}(X,\MC)} \\
&=\exp \left( \sum_{r \geq 1}  h_{(c)}(X,\MC)(y^r,x^r,z^r) \cdot \frac{t^r}{r} \right)
\end{split}
\end{equation}
and
\begin{equation} \label{gen-MHM2} \begin{split}
\sum_{n \geq 0} e_{(c)}(X^{(n)},\MC^{(n)})  \cdot t^n &= \prod_{p,q} 
\left(\frac{1}{1-y^px^qt}\right)^{ e_{(c)}^{p,q}(X,\MC)} \\
& = \exp \left( \sum_{r \geq 1}  e_{(c)}(X,\MC)(y^r,x^r) \cdot \frac{t^r}{r} \right) \:.
\end{split} \end{equation}
\end{theorem}

While the proof of Theorem \ref{main1} given in \cite{MSS} uses the K\"{u}nneth formula (\ref{Kue-mHs}) in a more direct way, 
our proof here will use an interpretation of (\ref{Kue-mHs}) in terms of pre-lambda rings. This more abstract point of view has the advantage
of being applicable to other situations as well (as discussed later on).\\ 

The definition of the $E$-polynomial and Hodge numbers $h_{(c)}^{p,q,k}(X,\MC)$
uses both the Hodge and the weight filtration of the mixed Hodge structure of $H_{(c)}^k(X,\MC)$,
and it is known that these invariants can't be generalized to suitable characteristic classes (see \cite{BSY}).
For characteristic class versions one has to work only with the Hodge filtration $F$ and  the corresponding
$\chi_y$-genus in $\Z[y^{\pm 1}]$:
\begin{equation} \chi^{(c)}_{-y}(X,\MC):=\sum_{p} f^p_{(c)}\cdot y^p,
\quad \text{with} \quad 
f^p_{(c)} :=\sum_{i} (-1)^i {\rm dim}_{\C} {\rm Gr}^p_F H^i_{(c)}(X,\MC)\:.
\end{equation}
Then $f^p_{(c)}=\sum_{q} e^{p,q}_{(c)}$ and  
$\chi^{(c)}_{-y}(X,\MC)=e_{(c)}(X,\MC)(y,1)$, so Theorem \ref{main1}
implies the following (compare also with \cite{MSS}[Cor.3])
\begin{cor} \label{cor-main1}
Let $X$ be a complex quasi-projective variety and $\MC \in D^b\mh(X)$ a bounded complex of mixed Hodge modules on $X$. With the above notations, the following formula holds:
\begin{equation}\label{eq:cor-main1} \begin{split}
\sum_{n \geq 0} \chi^{(c)}_{-y}(X^{(n)},\MC^{(n)})  \cdot t^n &= \prod_{p} 
\left(\frac{1}{1-y^pt}\right)^{ f_{(c)}^{p}(X,\MC)} \\
& = \exp \left( \sum_{r \geq 1}  \chi^{(c)}_{-y^r}(X,\MC) \cdot \frac{t^r}{r} \right) \:.
\end{split} \end{equation}
\end{cor}
This is in fact the formula which has been recently  generalized to a characteristic class version in \cite{CMSSY}[Thm.1.1].\\

For the constant Hodge module (complex) $\MC=\Q_X^H$ we get back (\ref{Che-betti}) and
(\ref{Che-e}) since, as shown in \cite{MSS}[Rem.2.4(i)], one has 
\begin{equation}\label{qp} \left(\Q_X^H\right)^{(n)}=\Q^H_{X^{(n)}}\:, \end{equation}
and Deligne's and Saito's mixed Hodge structure on $H^*_{(c)}(X,\Q)$ agree (\cite{Sa3}).
Here we only need this deep result for $X$ quasi-projective, where it quickly follows from the construction.\\ 

But of course we can use other coefficients, e.g., for $X$ pure dimensional we can use
a corresponding (shifted) intersection cohomology  mixed Hodge module 
$${IC'}_X^H= {IC}_X^H[-{\rm dim}_{\C}X] \in \mh(X)[-{\rm dim}_{\C}X]\subset D^b\mh(X)$$
calculating the intersection (co)homology $IH^*_{(c)}(X)=H^*_{(c)}(X,{IC'}_X)$ of $X$. 
Then we have similarly that (cf. \cite{MSS}[Rem.2.4(ii)]) 
\begin{equation}\label{ihp} \left({IC'}_X^H\right)^{(n)}={IC'}^H_{X^{(n)}}\:.\end{equation}

More generally, if $\LL$ is an admissible (graded polarizable) variation of mixed Hodge structures with quasi-unipotent monodromy at infinity defined on a smooth pure dimensional quasi-projective variety 
$U$, then $\LL$ corresponds by Saito's work to a shifted mixed Hodge module
$$\LL^H \in \mh(U)[-{\rm dim}_{\C}U]\subset D^b\mh(U)\:.$$
Note that the projection $p_n: U^n\to U^{(n)}$ is a finite ramified covering branched
along the ``fat diagonal'', i.e., the induced map of the configuration spaces on $n$ (un)ordered points in $U$:
$$p_n: F(U,n):=\{ (x_1, x_2, \dots, x_n) \in U^n \  \vert \ x_i \neq x_j \ \ {\rm for } \ \ i \neq j \} \to F(U,n)/\Sigma_n=:B(U,n)$$
is a finite unramified covering. Therefore $(\LL^H)^{(n)}$ is also a shifted mixed Hodge module
with $(\LL^H)^{(n)}|_{B(U,n)}$ corresponding to an admissible variation of mixed Hodge structures
on $B(U,n)$ (as before). For $U\subset X$ a Zariski open dense subset of a quasi-projective variety $X$, one can extend $\LL^H$ and $(\LL^H)^{(n)}|_{B(U,n)}$ uniquely to twisted intersection cohomology mixed Hodge modules
$${IC}_X^H(\LL) \in \mh(X) \quad \text{and} \quad {IC}_{X^{(n)}}^H(\LL^{(n)}) \in \mh(X^{(n)})\:.$$
As before, the shifted complexes 
$${IC'}_X^H(\LL):={IC}_X^H(\LL)[-{\rm dim}_{\C}X] \quad \text{and} \quad
{IC'}_{X^{(n)}}^H(\LL^{(n)}):={IC}_{X^{(n)}}^H(\LL^{(n)})[-n\cdot {\rm dim}_{\C}X] $$
calculate the corresponding twisted intersection (co)homology. And these are related by (cf. \cite{MSS}[Rem.2.4(ii)])
\begin{equation}\label{ihp-twist} \left({IC'}_X^H(\LL)\right)^{(n)}={IC'}^H_{X^{(n)}}(\LL^{(n)}) \:.\end{equation}

Thus we get all the generating series formulae also for these (twisted) intersection (co)homology invariants, and in particular for  
\begin{equation}\label{IC-inv}
{I\chi_{-y}^{(c)}}(X,\LL):={\chi_{-y}^{(c)}}(X,{IC'}_X^H(\LL)) \:.
\end{equation}
And this polynomial unifies the following invariants:
\begin{enumerate}
\item[(y=1)] 
${I\chi_{-1}^{(c)}}(X,\LL)={\chi^{IH}_{(c)}}(X,\LL)$ is the corresponding  intersection (co)homology Euler characteristic (with compact support).
\item[(y=0)] 
$I\chi_{0}^{(c)}(X,\LL)={I\chi_a^{(c)}}(X,\LL)$ is the corresponding intersection (co)homology  arithmetic genus (with compact support).
\item[(y=-1)] Assume finally that $X$ is projective of pure even complex dimension $2m$, with
$\LL$ a polarizable variation of pure Hodge structures with quasi-unipotent monodromy at infinity defined on a smooth Zariski open dense subset
$U\subset X$. If $\LL$ is of even weight, then the middle intersection (co)homology group
$$IH^m(X,\LL)=H^m(X,{IC'}_X(\LL))$$
gets an induced {\em symmetric} (non-degenerate) intersection form, with
$$I\chi_{1}(X,\LL)=\sigma(X,\LL):=\sigma\left(IH^m(X,\LL)\right)$$ the corresponding 
Goresky-MacPherson (twisted intersection homology) signature (\cite{GM1}).
This identification follows from Saito's {\em Hodge index theorem} for $IH^*(X,\LL)$ (\cite{Sa1}[Thm.5.3.1]), but see also \cite{MSS}[Sec.3.6]. 
\end{enumerate}

Here we only formulate the following  
\begin{cor}
Let $X$ be a pure dimensional complex quasi-projective variety, with
$\LL$ an admissible (graded polarizable) variation of mixed Hodge structures with quasi-unipotent monodromy at infinity defined on a smooth Zariski open dense subset of $X$.
Then
\begin{gather}\label{IC-genera} 
\sum_{n \geq 0} {I\chi^{(c)}_{-y}}(X^{(n)},\LL^{(n)})\cdot t^n 
= \exp \left( \sum_{r \geq 1}  {I\chi^{(c)}_{-y^r}}(X,\LL) \cdot \frac{t^r}{r} \right) \:.\\
\sum_{n \geq 0} {\chi^{IH}_{(c)}}(X^{(n)},\LL^{(n)})\cdot t^n 
= \exp \left( \sum_{r \geq 1} {\chi^{IH}_{(c)}}(X,\LL) \cdot \frac{t^r}{r} \right) = (1-t)^{-{\chi^{IH}_{(c)}}(X,\LL)} \:.\\
\sum_{n \geq 0} {I\chi_a^{(c)}}(X^{(n)},\LL^{(n)})\cdot t^n 
= \exp \left( \sum_{r \geq 1} {I\chi_a^{(c)}}(X,\LL) \cdot \frac{t^r}{r} \right) = (1-t)^{- {I\chi_a^{(c)}}(X,\LL)} \:.
\end{gather}
Assume in addition that $X$ is projective of even complex dimension,
with $\LL$ a polarizable variation of pure Hodge structures of even weight. Then 
\begin{equation}\label{IC-sig}
\sum_{n \geq 0} {\sigma}(X^{(n)},\LL^{(n)}) \cdot t^n=\frac{(1+t)^{\frac{\sigma(X,\LL)-\chi^{IH}(X,\LL)}{2}}}
{(1-t)^{\frac{\sigma(X,\LL)+\chi^{IH}(X,\LL)}{2}}}\:. 
\end{equation}
\end{cor}
The right-hand side of equation (\ref{IC-sig}) is a rational function in $t$ since, by Poincar\'e duality for twisted intersection homology, the Goresky-MacPherson signature and the intersection homology Euler characteristic have the same parity. 
In the special case of a projective rational homology manifold $X$ one has the isomorphism ${IC'}_X^H\simeq \Q_X^H$, whence $IH^*(X)=H^*(X)$, so we get back in this context the result (\ref{Za-sig})
by Hirzebruch and Zagier \cite{Za}.
For more general versions of formula (\ref{IC-sig}) see also \cite{MSS}.\\

Note that our coefficients $\MC\in D^b\mh(X)$, e.g., $\Q_X^H$ or ${IC'}_X^H(\LL)$, are in general highly complicated objects. For this reason, we give a proof of our results based only 
on suitable abstract formal properties of the category $D^b\mh(X)$, 
all of which are contained in the very deep work of M. Saito (\cite{Sa, Sa0, Sa1}) together with the recent paper \cite{MSS}.
The key underlying structures
become much better visible in this abstract context,
and this method of proof of Theorem \ref{main1} also applies to other situations, e.g., it yields the following closely related results:
\begin{theorem}\label{main2} Assume
\begin{enumerate}
\item[(a)] $X$ is a complex algebraic variety or a compact complex analytic set
(or a real semi-algebraic or compact subanalytic set, or a compact Whitney stratified set).
Let $\FC\in D^b_c(X)$ be a bounded complex of sheaves of vector spaces over a field $\K$ of characteristic zero, which is {\em constructible} (in the corresponding sense).
Then $H^*_{(c)}(X,\FC)$ is finite dimensional (\cite{Sch}), so that the corresponding Euler
characteristic 
$$\chi_{(c)}(X,\FC):= \chi\left( H^*_{(c)}(X,\FC) \right)$$
is defined. Then the same is true for $\chi_{(c)}(X^{(n)},\FC^{(n)})$, and
\begin{equation}
\sum_{n \geq 0} {\chi_{(c)}}(X^{(n)},\FC^{(n)})\cdot t^n =(1-t)^{-{\chi_{(c)}}(X,\FC)}
= \exp \left( \sum_{r \geq 1} {\chi_{(c)}}(X,\FC) \cdot \frac{t^r}{r} \right) \:.
\end{equation}
\item[(b)] $X$ is a compact complex algebraic variety or analytic set,
and $\FC\in D^b_{coh}(X)$ is a bounded complex of sheaves of $\OO_X$-modules with
{\em coherent} cohomology sheaves. Then $H^*(X,\FC)$ is finite dimensional by Serre's finiteness theorem.
So one can define 
$$\chi_a(X,\FC):=\chi\left( H^*(X,\FC) \right)$$
and similarly for $\chi_a(X^{(n)},\FC^{(n)})$. Then
\begin{equation}
\sum_{n \geq 0} {\chi_a}(X^{(n)},\FC^{(n)})\cdot t^n =(1-t)^{-{\chi_a}(X,\FC)}
= \exp \left( \sum_{r \geq 1} {\chi_a}(X,\FC) \cdot \frac{t^r}{r} \right) \:.
\end{equation}
\end{enumerate}
\end{theorem}

The proof of our results, as discussed in Section \ref{sec:pre-lambda}, 
relies on exploring and understanding
the relation between generating series for suitable invariants and {\em pre-lambda structures} on the coefficient theory on a point space. This leads us to a conceptually quick proof of our results. Moreover, it also suggests to consider the corresponding \emph{alternating projector} 
$$\PC^{alt}=(-)^{sign-\Sigma_n}:=
 \frac{1}{n!}\sum_{\sigma \in \Sigma_n} (-1)^{sign(\sigma)}\cdot \psi_{\sigma}$$ 
onto the {\em alternating} $\Sigma_n$-equivariant sub-object (with $\psi_{\sigma}$ given by the $\Sigma_n$-action as  before).\\

For a mixed Hodge module complex $\MC\in D^b\mh(X)$ on the complex quasi-projective variety $X$,
we define the corresponding {\em alternating object} 
\begin{equation}\label{alt-proj}
\MC^{\{n\}}:=\PC^{alt}({p_n}_*\MC^{\boxtimes{n}})= ({p_n}_*\MC^{\boxtimes{n}})^{sign-\Sigma_n} 
\in D^b\mh(X^{(n)}) \:.
\end{equation}
If, moreover, $rat(\MC)$ is a constructible sheaf (sitting in degree zero, and not a sheaf complex),
the following additional equivalent properties hold:
\begin{equation} \label{alt-MHM}
j_!j^*\MC^{\{n\}} \simeq \MC^{\{n\}} \quad \text{and} \quad i^*\MC^{\{n\}}\simeq 0 \:,
\end{equation}  
were $j: B(X,n)=X^{\{n\}}:= F(X,n)/\Sigma_n \to X^{(n)}$ is the open inclusion of the configuration space
$B(X,n)=X^{\{n\}}$ of all unordered $n$-tuples of different points in $X$, and $i$ is the closed inclusion of the complement of $X^{\{n\}}$ into $X^{(n)}$.
In fact, as in the proof of \cite{CMSS}[Lem.5.3], it is enough to show the corresponding statement $i^*\FC^{\{n\}}\simeq 0$
for the underlying constructible sheaf $\FC:=rat(\MC)$, since the alternating projector commutes with $rat$
(see \cite{MSS}). This vanishing can be checked on stalks. Finally, the stalk
of $\FC^{\{n\}}$ at a point $\bar{x}=(x_1,\dots,x_n)\in X^{(n)}\backslash B(X,n)$
is given as the image of the alternating projector acting on
$$({p_n}_*\FC^{\boxtimes{n}})_{\bar{x}}\simeq \oplus_{\{y|\; p_n(y)=\bar{x}\}} \: \FC^{\boxtimes n}_{y} \:,$$
and  for each $y\in p_n^{-1}(\{\bar{x}\})$ there is a transposition
$\tau \in \Sigma_n$ fixing $y$. This implies the claim.\\

So under these assumptions on $\MC$, one has a  K\"{u}nneth formula 
\begin{equation} \label{Kue-mHs-alt}
H^*_{c}(X^{\{n\}},\MC^{\{n\}})\simeq (H^*_{c}(X^n,\MC^{\boxtimes n}))^{sign-\Sigma_n} \simeq 
((H^*_{c}(X,\MC))^{\otimes{n}})^{sign-\Sigma_n} 
\end{equation}
for the cohomology with compact support. Again these are isomorphisms of graded groups of mixed Hodge structures 
and, for suitable choices, $\MC^{\{n\}}|_{X^{\{n\}}}$ becomes a highly interesting object on the configuration space
$B(X,n)=X^{\{n\}}$ of all unordered $n$-tuples of different points in $X$. \\

For example, $(\Q^H_X)^{\{n\}}|_{X^{\{n\}}}$ is a mixed Hodge module complex, whose underlying
rational sheaf complex is just the rank-one locally constant sheaf $\epsilon_n$ on $X^{\{n\}}$,
corresponding to the sign-representation of $\pi_1(X^{\{n\}})$ induced by the quotient homomorphism
$\pi_1(X^{\{n\}})\to \Sigma_n$ of the Galois covering $F(X,n)\to X^{\{n\}}$.
So for $X$ smooth we can think of $(\Q^H_X)^{\{n\}}|_{X^{\{n\}}}$ just as the corresponding variation of pure Hodge structures $\epsilon_n$ (of weight $0$) on the smooth variety
$X^{\{n\}}$. Similarly, for a smooth quasi-projective variety $X$ and
$\LL$  an admissible (graded polarizable) variation of mixed Hodge structures on $X$, with quasi-unipotent monodromy at infinity, one has on the smooth variety
$X^{\{n\}}$ the equality:
\begin{equation}\label{ihp-alt} 
\LL^{\{n\}}=\epsilon_n\otimes \LL^{(n)} \:, \end{equation}
These examples provide interesting coefficients that can be used in the following result (for the special case of constant coefficients
$\MC =\Q^H_X$ compare with \cite{Ge1}[Cor.5.7]).

\begin{cor} \label{MHM-alt}
Let $X$ be a complex quasi-projective variety and $\MC \in D^b\mh(X)$ a bounded complex of mixed Hodge modules on $X$ such that
$rat(\MC)$ is a constructible sheaf. 
Then the following formulae hold for the generating series of invariants of configuration spaces $X^{\{n\}}$ of unordered distinct points in $X$:
\begin{equation} \label{gen-MHM1-alt} 
\begin{split}
 \sum_{n \geq 0} h_{c}(X^{\{n\}},\MC^{\{n\}}) \cdot t^n  
 &=\prod_{p,q,k} \left(1+y^px^qz^kt\right)^{(-1)^k\cdot h_{c}^{p,q,k}(X,\MC)}\\
&= \exp \left( -\sum_{r \geq 1}  h_{c}(X,\MC)(y^r,x^r,z^r) \cdot \frac{(-t)^r}{r} \right) \:.
\end{split}
\end{equation}
\begin{equation} \label{gen-MHM2-alt} \begin{split}
\sum_{n \geq 0} e_{c}(X^{\{n\}},\MC^{\{n\}})  \cdot t^n &= \prod_{p,q} 
\left(1+y^px^qt\right)^{ e_{c}^{p,q}(X,\MC)} \\
& = \exp \left( -\sum_{r \geq 1}  e_{c}(X,\MC)(y^r,x^r) \cdot \frac{(-t)^r}{r} \right) \:.
\end{split} \end{equation}
\begin{equation}\label{eq:cor-main1-alt} \begin{split}
\sum_{n \geq 0} \chi^{c}_{-y}(X^{\{n\}},\MC^{\{n\}})  \cdot t^n &= \prod_{p} 
\left(1+y^pt\right)^{ f_{c}^{p}(X,\MC)} \\
& = \exp \left( -\sum_{r \geq 1}  \chi^{c}_{-y^r}(X,\MC) \cdot \frac{(-t)^r}{r} \right) \:.
\end{split} \end{equation}
\end{cor}

Abstracting the properties of  $D^b\mh(X)$ needed in our proof of Theorem \ref{main1}, we come to the following assumptions
used in the formulation and proof of our main result below, which unifies all our previous results: 

\begin{ass}\label{ass4}
\begin{enumerate}
 \item[(i)] Let $(-)_*$ be a (covariant) pseudo-functor on the category of complex quasi-projective varieties (with proper morphisms), taking values in a
pseudo-abelian (also called Karoubian) $\Q$-linear additive category $A(-)$.
\item[(ii)] For any quasi-projective variety $X$ and all $n$ there is a multiple external product 
$ \boxtimes^n:  \times^n\;A(X) \to A(X^n)$, equivariant with respect to a permutation action of $\Sigma_n$,
i.e., $M^{\boxtimes n}\in A(X^n)$ is a $\Sigma_n$-equivariant object for all $M\in A(X)$.
\item[(iii)] $A(pt)$ is endowed with a $\Q$-linear tensor structure $\otimes$, which makes it into a symmetric monoidal category.
\item[(iv)]  For any quasi-projective variety $X$, $M\in A(X)$ and all $n$, there is $\Sigma_n$-equivariant isomorphism
$k_*(M^{\boxtimes n})\simeq (k_*M)^{\otimes n}$, with $k$ the constant morphism to a point $pt$. Here, the $\Sigma_n$-action on the left-hand side
is induced from (ii), whereas the one on the right-hand side comes from (iii).
\end{enumerate}\end{ass}
Properties (i) and (ii) allow us to define for $M\in A(X)$ the \emph{symmetric} and \emph{alternating powers} $M^{(n)}, M^{\{n\}}\in A(X^{(n)})$
as in (\ref{sym-proj}) and (\ref{alt-proj}). For $X$ a point $pt$,  properties (i) and (iii) endow the Grothendieck group (with respect to direct sums)
$\bar{K}_0(A(pt))$  with a pre-lambda ring structure defined by (compare \cite{Hl}):
\begin{equation} 
\sigma_t: \bar{K}_0(A(pt))\to \bar{K}_0(A(pt))[[t]]\:;\:\:
[\VV] \mapsto 1+ \sum_{n\geq 1}\; [(\VV^{\otimes n})^{\Sigma_n}] \cdot t^n \:,
\end{equation}
with the \emph{opposite pre-lambda structure} induced by the alternating powers $[(\VV^{\otimes n})^{alt-\Sigma_n}]$. 

We can now state the following abstract generating series formula:
\begin{theorem}\label{adams-thm}
Under the above assumptions, the following  holds in $\bar{K}_0(A(pt))\otimes_{\Z}\Q[[t]]$,
for $X$ a quasi-projective variety and 
any $M\in ob(A(X))$:
\begin{equation}\label{adams1}
1+\sum_{n\geq 1}\;  [k_*M^{(n)}] \cdot t^n = 
\exp\left( \sum_{r\geq 1}\;\Psi_r([k_*M]) \cdot \frac{t^r}{r}\right) 
\end{equation}
and
\begin{equation}\label{adams2}
1+\sum_{n\geq 1}\; [k_*M^{\{n\}}] \cdot t^n = 
\exp\left( -\sum_{r\geq 1}\;\Psi_r([k_*M]) \cdot \frac{(-t)^r}{r}\right) \:,
\end{equation}
 with $\Psi_r$ denoting the $r$-th \emph{Adams operation} of the pre-lambda ring $\bar{K}_0(A(pt))$.
\end{theorem}

Property (iv) is used in the proof of this result (as given in Section \ref{sec:pre-lambda}) as a substitute for the K\"{u}nneth formula 
(\ref{Kue-mHs}). As we explain in the next section, these abstract generating series formulae 
translate into the concrete ones mentioned before, upon application 
 of certain pre-lambda ring homomorphisms.
Additionally, in Section \ref{Adams} we follow Atiyah's approach \cite{A} to power operations in $K$-theory.
In this way, we can extend  operations like symmetric or alternating powers from operations on the ``coefficient pre-lambda ring'' to a method of introducing 
very interesting coefficients on the symmetric products $X^{(n)}$, e.g. besides  $M^{(n)},M^{\{n\}}$,
we also get such coefficients related to the corresponding {\em Adams operations} $\Psi_r$.\\ 

Note that our Assumptions \ref{ass4} above are fullfilled for $A(X)=D^b\mh(X)$, viewed as a pseudo-functor with respect to either of the push-forwards $(-)_*$ or $(-)_!$.
Here (i) and (iii) follow from \cite{Sa}, e.g. (iii) follows form the equivalence of categories between $\mh(pt)$ and $\ms^p$ (as already mentioned before),
whereas (ii) and (iv) follow from our recent paper \cite{MSS}.
Finally, in the Appendix we collect in an abstract form some categorical notions needed to formulate a relation between {\em exterior products} and
{\em equivariant objects}, which suffices to show that our assumptions (i)-(iv) above are also fullfilled in many other situations, e.g.,
for the derived categories $D^b_c(X)$ and $D^b_{coh}(X)$.

\section{Symmetric products and pre-lambda rings}\label{sec:pre-lambda}
\subsection{Abstract generating series and Pre-lambda rings}

We work on an underlying geometric category $space$ of spaces (with finite products $\times$ and terminal object $pt$), e.g., the category of complex quasi-projective varieties (or compact topological or complex analytic spaces). We also assume that for all $X$ and $n$, the projection morphism
$p_n: X^n\to X^{(n)}=X^n/\Sigma_n$ to the $n$-th symmetric product exists in our category of spaces (e.g., as in the previously mentioned examples).
Let $(-)_*$ be a (covariant) pseudo-functor on this category of spaces, taking values in a
pseudo-abelian (also called Karoubian) $\Q$-linear additive category $A(-)$,
satisfying the Assumptions \ref{ass4}.\\

For a given space $X$ and $M\in A(X)$, 
the $n$-th exterior product $M^{\boxtimes n}\in A(X^n)$ underlies by Assumption \ref{ass4}(ii)
a $\Sigma_n$-equivariant object (with respect to the pseudo-functor $(-)_*$)
$$M^{\boxtimes n}\in A_{\Sigma_n}(X^n)$$
on $X^n$. Here $A_{\Sigma_n}(X^n)$ is the category of $\Sigma_n$-equivariant objects in $A(X^n)$,
as defined in the Appendix. By functoriality, one gets the induced  $\Sigma_n$-equivariant object
$$p_{n*}M^{\boxtimes n} \in A_{\Sigma_n}(X^{(n)}) \:.$$
Since $\Sigma_n$ acts trivially on $X^{(n)}$, this corresponds to an action of $\Sigma_n$ on $p_{n*}M^{\boxtimes n}$ in $A(X^{(n)})$, i.e., a group homomorphism $\psi: \Sigma_n \to Aut_{A(X^{(n)})}(p_{n*}M^{\boxtimes n})$ (see the Appendix for more details).
Since $A(X^{(n)})$ is a $\Q$-linear additive category, this allows us to define the projectors
$$\PC^{sym}:=(-)^{\Sigma_n} = \frac{1}{n!}\sum_{\sigma \in \Sigma_n} \psi_{\sigma}
\quad \text{and} \quad \PC^{alt}=(-)^{sign-\Sigma_n}:=
 \frac{1}{n!}\sum_{\sigma \in \Sigma_n} (-1)^{sign(\sigma)}\cdot \psi_{\sigma}\,$$
acting on elements in $A_{\Sigma_n}(X^{(n)})$. 
Finally, the ``pseudo-abelian" (or Karoubian) structure is used to define the 
\emph{symmetric} and resp. \emph{alternating} powers of $M$:
\begin{equation}M^{(n)}:=\PC^{sym}({p_n}_*M^{\boxtimes{n}})  \quad \text{resp.} ,\quad M^{\{n\}}:=\PC^{alt}({p_n}_*M^{\boxtimes{n}}) 
\in A(X^{(n)}) \:.\end{equation} 

The generating series we are interested in codify invariants of $k_*M^{(n)}$ and $k_*M^{\{n\}}$, respectively, with $k$
the constant morphism to the terminal object $pt\in ob(space)$. By Assumption \ref{ass4}(iv), we have that
\begin{equation}\label{Kue-pt}
 k_{*}(M^{\boxtimes n}) \simeq (k_*M)^{\otimes n} \in A_{\Sigma_n}(pt)\:.
\end{equation}
Note that the image of an object acted upon by a projector $\PC$ as above is {\it functorial} under $\Sigma_n$-equivariant morphisms. For example, if we let 
$k: X^{(n)} \to pt$ be the constant map to a point,   then for any $M \in A(X)$ we have: 
\begin{equation} \label{k*!}
k_*(\PC({p_n}_*M^{\boxtimes{n}}))=\PC(k_*{p_n}_*M^{\boxtimes{n}}) \simeq \PC\left( (k\circ p_n)_*M^{\boxtimes{n}}\right) \simeq
\PC\left((k_*M)^{\otimes n}\right) \:.
\end{equation}
In particular, by using $\PC=\PC^{sym}$ and resp. $\PC=\PC^{alt}$, we get the following key isomorphisms in $A(pt)$ (abstracting the K\"{u}nneth formulae
(\ref{Kue-mHs}) and (\ref{Kue-mHs-alt}) from the introduction):
\begin{equation}\label{key}
k_*(M^{(n)}) \simeq \left((k_*M)^{\otimes n}\right)^{\Sigma_n} \quad \text{resp.} , \quad
k_*(M^{\{n\}}) \simeq \left((k_*M)^{\otimes n}\right)^{sign-\Sigma_n} \:.
\end{equation}
This allows the calculation of our invariants of $k_*M^{(n)}$ and $k_*M^{\{n\}}$, respectively, 
 in terms of those for $k_*M \in A(pt)$ and the symmetric monoidal structure $\otimes$. At this point, we note that the above arguments make use of all of our Assumptions \ref{ass4}.\\
 
Let $\bar{K}_0(-)$ denote the Grothendieck group of an additive
category viewed as an exact category by the split exact sequences corresponding to direct sums
$\oplus$, i.e., the Grothendieck group associated to the abelian monoid of isomorphism classes of objects with the direct sum. Then $\bar{K}_0(A(pt))$ becomes a commutative ring with unit $1_{pt}$ and product induced by $\otimes$.
Moreover, Assumptions \ref{ass4}(i) and (iii) endow  $\bar{K}_0(A(pt))$ with
 a canonical {\em pre-lambda structure} (\cite{Hl}):
\begin{equation} \label{pre-lambda}
\sigma_t: \bar{K}_0(A(pt))\to \bar{K}_0(A(pt))[[t]]\:;\:\:
[\VV] \mapsto 1+ \sum_{n\geq 1}\; [(\VV^{\otimes n})^{\Sigma_n}] \cdot t^n \:,
\end{equation}
with the \emph{opposite pre-lambda structure} $\lambda_t:=\sigma_{-t}^{-1}$ induced by the alternating powers $[(\VV^{\otimes n})^{alt-\Sigma_n}]$:
\begin{equation} \label{opp-pre-lambda}
\lambda_t: \bar{K}_0(A(pt))\to \bar{K}_0(A(pt))[[t]]\:;\:\:
[\VV] \mapsto 1+ \sum_{n\geq 1}\; [(\VV^{\otimes n})^{sign-\Sigma_n}] \cdot t^n \:.
\end{equation}
This opposite pre-lambda structure $\lambda_t$ is sometimes more natural,
e.g., it is  a lambda structure \cite{Hl}[Lem.4.1].
Recall here that a pre-lambda structure on a commutative ring $R$ with unit $1$ just means a group homomorphism
$$\sigma_t: (R,+)\to (R[[t]],\cdot)\:;\:\:
r \mapsto 1+ \sum_{n\geq 1}\; \sigma_n(r) \cdot t^n $$
with $\sigma_1=id_R$, where ``$\cdot$" on the target side denotes the multiplication of formal power series. So this corresponds to a family of self-maps $\sigma_n: R\to R$
($n\in \N_0$) satisfying for all $r\in R$:
$$\sigma_0(r)=1,\:
\sigma_1(r)=r \quad \text{and} \quad \sigma_k(r)=\sum_{i+j=k} \sigma_i(r)\cdot \sigma_j(r)\:.$$

For a pre-lambda ring $R$, there is the well-known formula
(cf. \cite{Kn,Ge1}):
\begin{equation} \label{adams}
\sigma_t(a) =\sum_{n\geq 0} \; \sigma_n(a) \cdot t^n = \exp\left( \sum_{r\geq 1}\;
\Psi_r(a) \cdot \frac{t^r}{r}\right) \in R\otimes_{\Z}\Q[[t]]\:,
\end{equation}
following from the definition of the corresponding $r$-th \emph{Adams operation} $\Psi_r$:
$$\lambda_t(a)^{-1}\cdot \frac{d}{dt}(\lambda_t(a))=:\frac{d}{dt}(\log(\lambda_t(a)) =:
\sum_{r\geq 1} (-1)^{r-1}\Psi_r(a)\cdot t^{r-1} \in R[[t]]$$
by
$$\sigma_{-t}(a)^{-1}=\lambda_t(a) =\exp\left( \sum_{r\geq 1}\;
(-1)^{r-1}\Psi_r(a) \cdot \frac{t^r}{r}\right) \in R\otimes_{\Z}\Q[[t]]\:.$$

Applying these formulae to the pre-lambda ring $R=\bar{K}_0(A(pt))$, we get the following 
\begin{proof}[Proof of Theorem \ref{adams-thm}]
Our key isomorphisms (\ref{key}) yield
for the element $a=[k_*M]\in \bar{K}_0(A(pt))$ the following identification of the generating series
in terms of the pre-lambda structure:
$$ 
1+\sum_{n\geq 1}\;  [k_*M^{(n)}] \cdot t^n = 1+ \sum_{n\geq 1}\; [(k_*M)^{\otimes n})^{\Sigma_n}] \cdot t^n
=\sigma_t\left([k_*M]\right)\:,$$
and, resp., 
$$1+\sum_{n\geq 1}\; [k_*M^{\{n\}}] \cdot t^n = 1+ \sum_{n\geq 1}\; [(k_*M)^{\otimes n})^{sign-\Sigma_n}] \cdot t^n
=\lambda_t\left([k_*M]\right)\:.$$
\end{proof}


\subsection{Examples and homomorphisms of pre-lambda rings}
In this section, we explain how suitable  homomorphisms of pre-lambda rings can be used to
translate our abstract generating series formulae of Theorem \ref{adams-thm} into the concrete ones mentioned in the introduction.
We start with some more examples of pre-lambda rings. \\

First, we have the Grothendieck group
$\bar{K}_0(A(pt))$ of a pseudo-abelian $\Q$-linear category $A(pt)$ with a tensor structure $\otimes$, which makes it into a symmetric monoidal category.
This includes our pseudo-functors $A(-)$ taking values in the $\Q$-linear triangulated categories $D^b\mh(-), D^b_c(-)$ and resp. $D^b_{coh}(-)$,
which are pseudo-abelian by \cite{BS, LC}. Similarly for the pseudo-functors $A(-)$ (with respect to finite morphisms) taking values in the $\Q$-linear abelian categories of mixed Hodge modules, perverse or constructible sheaves, and coherent sheaves.
In the last examples, we also have the following 

\begin{lemma} 
If $A(pt)$ as above is also an {\em abelian category} with $\otimes$ {\em exact in both variables},
the pre-lambda structure on $\bar{K}_0(A(pt))$ descends to one on the usual Grothendieck group
$$K_0(A(pt))=\bar{K}_0(A(pt))/(ex-seq)$$ 
with a direct sum for all {\em short exact sequences}.
\end{lemma}

\begin{proof}
For a short exact sequence 
$$0\to\VV'\to \VV \to \VV''\to 0 $$ in $A(pt)$, 
one can introduce an increasing two step filtration on $\VV$:
$$F: F^0\VV:=\VV' \subset F^1\VV=:\VV,\quad \text{with} \quad Gr^F_0\VV\simeq \VV',\:
Gr^F_1\VV\simeq \VV''\:.$$
Then by the exactness of $\otimes$ and of the projector $\PC=(-)^{\Sigma_n}$, one gets an induced filtration $F$
 on $(\VV^{\otimes n})^{\Sigma_n}$ with (compare \cite{De}[Part II, sec.1.1]):
$$Gr^F_*\left((\VV^{\otimes n})^{\Sigma_n}\right) \simeq 
\left((Gr^F_* \VV)^{\otimes n}\right)^{\Sigma_n}\:,$$
so that
$$[\left(\VV^{\otimes n}\right)^{\Sigma_n}] =[Gr^F_*\left((\VV^{\otimes n})^{\Sigma_n}\right)] = [\left((Gr^F_* \VV)^{\otimes n}\right)^{\Sigma_n}] \in K_0(A(pt))\:.$$
\end{proof}

Similarly, for the (symmetric bimonoidal) category $space$, one gets on the Grothendieck group 
$(\bar{K}_0(space),\cup)$
associated to the abelian monoid of isomorphism classes of objects with the sum coming from
the disjoint union $\cup$ (or categorical coproduct) and the product induced by $\times$, the structure of a commutative ring with unit $[pt]$ and zero $[\emptyset]$.
And for our category $space=var/\C$ of complex quasi-projective varieties (or compact topological or complex analytic spaces), we also get a pre-lambda ring structure on $\bar{K}_0(var/\C)$
by the {\em Kapranov zeta function} (\cite{Ka}):
\begin{equation} \label{Ka-zeta}
\sigma_t: \bar{K}_0(var/\C)\to \bar{K}_0(var/\C)[[t]]\:;\:\:
[X] \mapsto 1+ \sum_{n\geq 1}\; [X^{(n)}] \cdot t^n \:,
\end{equation}
because $(X\cup Y)^{(n)} \simeq \bigcup_{i+j=n} (X^{(i)} \times Y^{(j)})$.
In fact, this pre-lambda ring structure factorizes in the context of complex quasi-projective varieties also over the motivic Grothendieck group 
$$K_0(var /\C):= \bar{K}_0(var/\C)/(add) \:,$$
defined as the quotient by the {\em additivity relation}
$[X]=[Z]+ [X \backslash Z]$ for $Z\subset X$ a closed complex subvariety.\\

Finally, we have the following
\begin{example}\label{adams-laurent}
The Laurent polynomial ring $\Z[x_1^{\pm 1},\dots,x_n^{\pm 1}]$ in $n$
variables ($n\geq 0$) becomes a pre-lambda ring (e.g., see \cite{Go} or \cite{HRV}[p.578]) by
\begin{equation}\label{Laurant}
\sigma_t \left(\sum_{\vec{k}\in \Z^n} a_{\vec{k}}\cdot \vec{x}^{\;\vec{k}}\right)
:=\prod_{\vec{k}\in \Z^n} \left(1-\vec{x}^{\;\vec{k}}\cdot t \right)^{-a_{\vec{k}}} \:.
\end{equation}
The corresponding $r$-th Adams operation $\Psi_r$ is given by
(compare e.g. \cite{Go}):
$$\Psi_r(p(x_1,\cdots,x_n))=p(x_1^r,\cdots,x_n^r) \:,$$
e.g., with the notations from the introduction, we have
\begin{equation}\label{adams-e}\begin{split}
\Psi_r(h_{(c)}(X,\MC)(y,x,z))&= h_{(c)}(X,\MC)(y^r,x^r,z^r) \, \\ 
\Psi_r(e_{(c)}(X,\MC)(y,x))&= e_{(c)}(X,\MC)(y^r,x^r) \\ 
\text{and} \quad
\Psi_r(\chi^{(c)}_{-y}(X,\MC))&= \chi^{(c)}_{-y^r}(X,\MC) \:.
\end{split}
\end{equation}
\end{example}

We continue with examples of pre-lambda ring homomorphisms, which are needed for 
translating our abstract generating series formulae of Theorem \ref{adams-thm} into the concrete ones mentioned in the introduction.\\

Since the cohomolgy (with compact support) is {\em additive} under disjoint union, i.e., 
$$H^*_{(c)}(X\cup Y,\Q) = H^*_{(c)}(X,\Q) \oplus H^*_{(c)}(Y,\Q)\:,$$
in the context of complex quasi-projective varieties we get  a group homomorphism
$$h_{(c)}: \bar{K}_0(var/\C) \to \Z[y^{\pm 1},x^{\pm 1},z^{\pm 1}]\:;\:\:
[X]\mapsto \sum_{p,q,k} h_{(c)}^{p,q,k}(X)\cdot y^px^q(-z)^k \:,$$
with $h_{(c)}^{p,q,k}(X):=h^{p,q}(H^k_{(c)}(X,\Q))$.
Then by the usual K\"{u}nneth isomorphism (which respects the underlying mixed Hodge structures of Deligne),
$h_{(c)}$ becomes a ring homomorphism.
So the generating series (\ref{Che-betti}) for $h_{(c)}$ just tells us the well-known fact
(compare \cite{Che,GLM1,Go}) that $h_{(c)}$ is a morphism of pre-lambda rings.
And the corresponding morphism of pre-lambda rings $e_c$ factorizes over the motivic Grothendieck ring:
$$e_c: K_0(var/\C) \to \Z[y^{\pm 1},x^{\pm 1}]\:;\:\:
[X]\mapsto \sum_{p,q} e_{c}^{p,q}(X)\cdot y^px^q \:,$$
because the long exact sequence for the cohomology with compact support
$$\cdots \to H^k_c(X\backslash Z,\Q) \to H^k_c(X,\Q) \to H^k_c(Z,\Q) \to \cdots $$
for $Z\subset X$ a closed subvariety is an exact sequence of mixed Hodge structures.
(For applications of this to the associated {\em power structures} compare with \cite{GLM0,GLM1,Go}.)

We assert that these can be factorized  as homomorphisms of pre-lambda rings:
$$\begin{CD}
h: \bar{K}_0(var/\C) @> k_*(\Q^H_{?}) >> \bar{K}_0(D^b\mh(pt)) @> h >> 
\Z[y^{\pm 1},x^{\pm 1},z^{\pm 1}]
\end{CD}$$
and
$$\begin{CD}
h_c:\bar{K}_0(var/\C) @> k_!(\Q^H_{?}) >> \bar{K}_0(D^b\mh(pt)) @> h >> 
\Z[y^{\pm 1},x^{\pm 1},z^{\pm 1}]\\
@VVV @V can VV @VV z=1 V \\
e_c: K_0(var/\C) @> k_!(\Q^H_{?}) >> K_0(\mh(pt)) @> e >> 
\Z[y^{\pm 1},x^{\pm 1}] \:.
\end{CD}$$
Here $K_0(\mh(pt))$ is the Grothendieck group of the abelian category $\mh(pt)$,
with $can$ induced by the alternating sum of cohomology objects of a complex.
The fact that the group homomorphisms on the left side (compare \cite{BSY}[sec.5]):
$$[X]\mapsto [k_?(\Q^H_X)] \quad \text{ for $?=*,!$} $$ 
are ring homomorphisms of pre-lambda rings follows from (\ref{qp}) and (\ref{key}).\\

To show that  $h:  \bar{K}_0(D^b\mh(pt)) \to \Z[y^{\pm 1},x^{\pm 1},z^{\pm 1}]$ 
is a homomorphism of pre-lambda rings, we further factorize it as
\begin{equation}
\begin{CD}
 \bar{K}_0(D^b\mh(pt)) @> H^* >> \bar{K}_0(Gr^{-}(\mh(pt))) @> \sim >> \bar{K}_0(Gr^{-}(\ms^p)) \\
@V h VV   @VVV @VV {\rm forget} V \\
 \Z[y^{\pm 1},x^{\pm 1},z^{\pm 1}] @<< h < \bar{K}_0(Gr^{-}(Gr^2(\ve_{f}(\C)))) @< Gr_F^*Gr^W_* << \bar{K}_0(Gr^{-}(\ms)) \:,  
  \end{CD}
\end{equation}
where the following notations are used:
\begin{enumerate}
\item[(a)] For an additive (or abelian) tensor category $(A,\otimes)$,  $Gr^{-}(A)$  denotes the additive (or abelian) tensor category of {\em bounded graded}
objects in $A$, i.e., functors $G: \Z\to A$, with $G_n:=G(n)=0$ except for finitely many $n\in \Z$. Here, $$(G\otimes G')_n:= \oplus_{i+j=n} G_i\otimes G_j \:,$$
with the Koszul symmetry isomorphism (indicated by the $-$ sign in $Gr^{-}$): $$(-1)^{i\cdot j} s(G_i,G_j): G_i\otimes G_j \simeq  G_j\otimes G_i \:.$$
\item[(b)] $H^*: D^b\mh(pt)\to Gr^{-}(\mh(pt))$ is the total cohomology functor $\VV\mapsto \oplus_n H^n(\VV)$. Note that this is a functor of additive tensor categories
(i.e., it commutes with direct sums $\oplus$ and tensor products $\otimes$), if we choose the Koszul symmetry isomorphism on $Gr^{-}(\mh(pt))$.
In fact, $D^b\mh(pt)$ is a triangulated category with bounded $t$-structure satisfying \cite{Bi}[Def.4.2], so that the claim follows from \cite{Bi}[thm.4.1, cor.4.4].
\item[(c)] The isomorphism $\mh(pt)\simeq \ms^p$ from Saito's work \cite{Sa, Sa0} was already mentioned in the introduction.
\item[(d)] ${\rm forget}: \ms^p\to \ms$ is the functor of forgetting that the corresponding $\Q$-mixed Hodge structure is graded polarizable.
\item[(e)] $Gr_F^*Gr_*^W: \ms\to  Gr^2(\ve_{f}(\C))$ is the functor of taking the associated bigraded finite dimensional $\C$-vector space
$$\ms \ni \VV \mapsto \oplus_{p,q} \; Gr_F^pGr^W_{p+q}(\VV\otimes_{\Q}\C) \in  Gr^2(\ve_{f}(\C)) \:.$$
This is again a functor of additive tensor categories, if this time we use the induced symmetry isomorphism without any sign changes.
\item[(f)] $h:   Gr^{-}(Gr^2(\ve_{f}(\C)))\to \Z[y^{\pm 1},x^{\pm 1},z^{\pm 1}]$ is given by taking the dimension counting Laurent polynomial
$$\oplus (V^{p,q})^k \mapsto \sum_{p,q,k} \;dim((V^{p,q})^k) \cdot y^px^q(-z)^k \:,$$
with $k$ the degree with respect to the grading in $Gr^{-}$ (this fact corresponds to the sign choice of numbering by $(-z)^k$ in this
definition). 
\end{enumerate}

\begin{remark}\label{rem-kue}
The fact that total cohomology functor $H^*: D^b\mh(pt)\to Gr^{-}(\mh(pt))$ is a tensor functor corresponds to the  K\"{u}nneth formula 
$$H^*(\VV^{\otimes n}) \simeq (H^*(\VV))^{\otimes n} \ , \quad \text{ for $\VV\in D^b\mh(pt)$.}$$
Together with (\ref{key}) this
implies the important K\"{u}nneth isomorphism (\ref{Kue-mHs}) from the introduction. (For a more direct approach compare with \cite{MSS}.)
\end{remark}

Note that all functors in (a)-(e) above are functors of $\Q$-linear tensor categories, so they induce ring homomorphisms of the corresponding Grothendieck groups
$\bar{K}_0(-)$, respecting the corresponding pre-lambda structures (\ref{pre-lambda}). Therefore, 
$h: \bar{K}_0(D^b\mh(pt)) \to \Z[y^{\pm 1},x^{\pm 1},z^{\pm 1}]$ is a homomorphism of pre-lambda rings by the following 

\begin{prop}\label{prop:h}
The ring homomorphism $$h: \bar{K}_0(Gr^{-}(Gr^2(\ve_{f}(\C))))\to \Z[y^{\pm 1},x^{\pm 1},z^{\pm 1}] $$
is a homomorphism of pre-lambda rings.
\end{prop}

\begin{proof} Since $h$ is a ring homomorphism, it is enough to check the formula
 $$\sigma_t(h(L)) = 1+ \sum_{n\geq 1} \; h\left( (L^{\otimes n})^{\Sigma_n}\right) \cdot t^n$$
for the set of generators given by the one dimensional graded vector spaces $L=(L^{p,q})^k=:L^{p,q,k}$ sitting in degree $(p,q,k)$. 
Here we have to consider the two cases when $k$ is even and odd, respectively, each giving a different meaning to $((-)^{\otimes n})^{\Sigma_n}$ by the graded commutativity.
For $k$ even,  $((-)^{\otimes n})^{\Sigma_n}$ corresponds to  the $n$-th symmetric power $Sym_n(-)$, whereas for $k$ odd it corresponds to the $n$-th alternating power $Alt_n(-)$.
For $k$ even we get 
$$(L^{\otimes n})^{\Sigma_n}= Sym_n(L)=L^{\otimes n},
\quad \text{with} \quad h(L^{\otimes n})=(y^{p}x^{q}z^{k})^n,$$
thus  
$$1+ \sum_{n\geq 1} \; h\left( (L^{\otimes n})^{\Sigma_n}\right) \cdot t^n= 1+ \sum_{n\geq 1} \;(y^{p}x^{q}z^{k}t)^n = (1-y^{p}x^{q}z^{k}t)^{-1} \:.$$
For $k$ odd we get  
$$(L^{\otimes n})^{\Sigma_n}= Alt_n(L) = \begin{cases}
0 &\text{for $n>1$,}\\
L &\text{for $n=1$.}
                                                                                       \end{cases}$$
Therefore
$$1+ \sum_{n\geq 1} \; h\left( (L^{\otimes n})^{\Sigma_n}\right) \cdot t^n= 1- y^{p}x^{q}z^{k}t \:.$$
Then Proposition \ref{prop:h}  follows, since by (\ref{Laurant}) we have 
$$\sigma_t(1+y^px^qz^k)=(1-y^px^qz^kt)^{-1} \quad \text{and} \quad \sigma_t(1-y^px^qz^k)=1-y^px^qz^kt \:.$$
\end{proof}

By exactly the same method one also gets the following homomorphism of pre-lambda rings:
$$\begin{CD} 
  \bar{K}_0(D^b_c(pt)) @> H^* >>  \bar{K}_0(Gr^{-}(\ve_{f}(\C))) @> P >> \Z[z^{\pm 1}]
  \end{CD}$$
and, resp., 
$$\begin{CD} 
  \bar{K}_0(D^b_{coh}(pt)) @> H^* >>  \bar{K}_0(Gr^{-}(\ve_{f}(\C))) @> P >> \Z[z^{\pm 1}]\:.
  \end{CD}$$
Here $P: Gr^{-}(\ve_{f}(\C))\to \Z[z^{\pm 1}]$ is the Poincar\'e polynomial homomorphism
 given by taking the dimension counting Laurent polynomial
$$\oplus V^k \mapsto \sum_{k} \;dim(V^k) \cdot (-z)^k \:,$$
with $k$ the degree with respect to the grading in $Gr^{-}$.
Also,  $H^*$ is once more a functor of tensor categories by \cite{Bi}[thm.41, cor.4.4].
Similarly for the Euler characteristic homomorphism  $\chi=P(1): Gr^{-}(\ve_{f}(\C))\to \Z$,
where $\Z$ is endowed with the pre-lambda structure $\sigma_t(a)=(1-t)^{-a}$
(and the corresponding Adams operation $\Psi_r(a)=a$).\\ 

Using these homomorphisms of pre-lambda rings, we deduce the concrete generating series formulae from the introduction from our main  Theorem \ref{adams-thm}
in the following way:
\begin{itemize}
 \item Choosing the pseudo-functor $A(-)=D^b\mh(-)$ in Theorem \ref{adams-thm}, we get Theorem \ref{main1} by applying the homomorphism of pre-lambda rings
$h:  \bar{K}_0(D^b\mh(pt))$ $\to \Z[y^{\pm 1},x^{\pm 1},z^{\pm 1}]$ to (\ref{adams1}). Corollary \ref{MHM-alt} 
follows by applying $h$ to (\ref{adams2}).
\item Choosing the pseudo-functor $A(-)=D^b_c(-)$ in Theorem \ref{adams-thm}, we get Theorem \ref{main2}(a) by applying the homomorphism of pre-lambda rings
$\chi:  \bar{K}_0(D^b_c(pt)) \to \Z$ to (\ref{adams1}).
Of course, we get similar results by using the Poincar\'e polynomial homomorphism $P:  \bar{K}_0(D^b_c(pt)) \to \Z[z^{\pm 1}]$,
generalizing Macdonald's formula (\ref{Macd-betti}).
\item Choosing the pseudo-functor $A(-)=D^b_{coh}(-)$ (with respect to proper maps) in Theorem \ref{adams-thm}, we get Theorem \ref{main2}(b) by  applying the homomorphism of pre-lambda rings
$\chi:  \bar{K}_0(D^b_{coh}(pt)) \to \Z$ to (\ref{adams1}).
Again, we get similar results by using the Poincar\'e polynomial homomorphism $P:  \bar{K}_0(D^b_{coh}(pt)) \to \Z[z^{\pm 1}]$ instead.
\end{itemize}

\section{Adams operations for symmetric products}\label{Adams}
A natural way to extend  operations like symmetric or alternating powers from operations on the pre-lambda ring $\bar{K}_0(A(pt))$ (or $K_0(A(pt))$) to a method of introducing 
very important coefficients like $\MC^{(n)},\MC^{\{n\}}$ on the symmetric products $X^{(n)}$, 
follows Atiyah's approach \cite{A} to power operations in $K$-theory.
In particular, one gets coefficients related to the corresponding {\em Adams operations}
$\Psi_r$.

We continue to work on an underlying geometric category $space$ of spaces (with finite products $\times$ and terminal object $pt$), e.g., the category of complex quasi-projective varieties (or compact topological or complex analytic spaces). We also assume that for all $X$ and $n$, the projection morphism
$p_n: X^n\to X^{(n)}=X^n/\Sigma_n$ to the $n$-th symmetric product exists in our category of spaces (e.g., as in the  examples mentioned above).
Let $(-)_*$ be a (covariant) pseudo-functor on this category of spaces, taking values in a
pseudo-abelian $\Q$-linear additive category $A(-)$, and
satisfying the Assumptions \ref{ass4}.\\

Since $\Sigma_n$ acts trivially on the symmetric product $X^{(n)}$, and $A(X^{(n)})$ is a pseudo-abelian $\Q$-linear category,
one has the following decomposition 
(compare \cite{Ge1}[Thm.3.2])
\begin{equation}\label{K-decomp}
\bar{K}_0(A_{\Sigma_n}(X^{(n)})) \simeq \bar{K}_0(A(X^{(n)}))\otimes _{\Z} Rep_{\Q}(\Sigma_n) 
\end{equation}
for the Grothendieck group of the category $A_{\Sigma_n}(X^{(n)})$ of $\Sigma_n$-equivariant objects in $A(X^{(n)})$.
In fact this follows directly from the corresponding decomposition of
$\VV\in A_{\Sigma_n}(X^{(n)})$ by  {\em Schur functors} 
$S_{\lambda}:A_{\Sigma_n}(X^{(n)})\to A(X^{(n)}) $
(\cite{De2,Hl}):
$$\VV \simeq \sum_{\lambda \vdash n} \; V_{\lambda}\otimes S_{\lambda}(\VV)\:,$$
with $V_{\lambda}$ the irreducible $\Q$-representation of $\Sigma_n$ corresponding to the partition $\lambda$ of $n$.

This fact is used in the following definition of an {\em operation}
$$\phi_X: ob(A(X))\to \bar{K}_0(A(X^{(n)}))$$ 
associated to a group homomorphism
$$\phi \in  Rep_{\Q}(\Sigma_n)_*:=hom_{\Z}(Rep_{\Q}(\Sigma_n), \Z)$$ on the rational representation ring of $\Sigma_n$.

\begin{definition}
Let the group homomorphism $\phi \in  Rep_{\Q}(\Sigma_n)_*:=hom_{\Z}(Rep_{\Q}(\Sigma_n), \Z)$ be given.
Then we define the operation $\phi_X: ob(A(X))\to \bar{K}_0(A(X^{(n)}))$ as follows:
\begin{equation} \begin{split} &\begin{CD}
\phi_X: ob(A(X)) @> (-)^{\boxtimes n} >> \bar{K}_0(A_{\Sigma_n}(X^{n})) 
@> p_{n*} >> \bar{K}_0(A_{\Sigma_n}(X^{(n)}))\end{CD}\\
& \begin{CD}
\simeq  \bar{K}_0(A(X^{(n)}))\otimes _{\Z} Rep_{\Q}(\Sigma_n) @> id \otimes \phi >>
\bar{K}_0(A(X^{(n)}))\otimes \Z \simeq \bar{K}_0(A(X^{(n)}))\:.
\end{CD} \end{split}
\end{equation} \end{definition}

If $X$ is a point $pt$, this definition
of the operation $\phi_{pt}$ can be reformulated by Assumption \ref{ass4}(iv) as
$$ \begin{CD}
\phi_{pt}: ob(A(pt)) @> (-)^{\otimes n} >> \bar{K}_0(A_{\Sigma_n}(pt)) 
\simeq  \bar{K}_0(A(pt))\otimes _{\Z} Rep_{\Q}(\Sigma_n) @> id \otimes \phi >>
 \bar{K}_0(A(pt))\:.
\end{CD}
$$

\begin{prop}\label{trace}
For $X=pt$ one gets an induced self-map $$\phi_{pt}: \bar{K}_0(A(pt))\to \bar{K}_0(A({pt}))$$ on the  Grothendieck group $\bar{K}_0(A(pt))$.
\end{prop}

\begin{proof}
Since $A(-)$ takes values in a Karoubian $\Q$-linear category,  it follows that for a space $Z$ with a trivial $G$-action and for subgroups $G'\subset G\subset \Sigma_n$, one has a $\Q$-linear 
{\em induction functor} (compare e.g., \cite{Ge1}) 
\begin{equation}\label{ind}
Ind^{G}_{G'}: A_{G'}(Z)\to A_{G}(Z):\; \VV\mapsto  Ind^{G}_{G'}(\VV):= (\Q[G]\boxtimes \VV)^{G'}\:, 
\end{equation}
so that it agrees on the level of Grothendieck groups with the map
$$Ind^{G}_{G'}:  \bar{K}_0(A_{G'}(Z)) \simeq \bar{K}_0(A(Z))\otimes _{\Z} Rep_{\Q}(G') \to \bar{K}_0(A(Z))\otimes _{\Z} Rep_{\Q}(G) \simeq  
 \bar{K}_0(A_G(Z))$$
coming from the usual induction homomorphism $Ind^{G}_{G'}: Rep_{\Q}(G') \to  Rep_{\Q}(G)$.
Then for $\VV,\VV' \in A(pt)$ one has an isomorphism (\cite{De2}):
$$(\VV\oplus \VV')^{\otimes n} \simeq \oplus_{i+j=n}\; Ind^{\Sigma_n}_{\Sigma_i\times \Sigma_j}\left(\VV^{\otimes i}\otimes \VV'^{\otimes j}\right) \:,$$
so that the total power map
\begin{gather*}
ob(A(pt)) \to 1+\sum_{n\geq 1}\; \bar{K}_0(A(pt))\otimes _{\Z} Rep_{\Q}(\Sigma_n) \cdot t^n \subset \left(\bar{K}_0(A(pt))\otimes _{\Z} Rep_{\Q}(\Sigma)\right)[[t]]\;:\\
\VV\mapsto 1+\sum_{n\geq 1} \; [\VV^{\otimes n}]\cdot t^n
\end{gather*}
induces a morphism of commutative semigroups, from the isomorphism classes of objects in $ ob(A(pt)$ with the direct sum $\oplus$, to the group of special  invertible 
power series in $$1+\sum_{n\geq 1}\; \bar{K}_0(A(pt))\otimes _{\Z} Rep_{\Q}(\Sigma_n) \cdot t^n$$ with the power series multiplication.
Here, 
$$ Rep_{\Q}(\Sigma):=\oplus_{n\geq 0}\;  Rep_{\Q}(\Sigma_n)$$
is the total (commutative) representation ring (compare e.g., \cite{Kn}) with the {\em cross  product}
$$\boxtimes: Rep_{\Q}(\Sigma_i)\times Rep_{\Q}(\Sigma_j) \to Rep_{\Q}(\Sigma_{i+j}):\; (\VV,\VV')\mapsto Ind^{\Sigma_{i+j}}_{\Sigma_i\times \Sigma_j}(\VV\otimes \VV') \:. $$
So the total power map induces a group homomorphism
$$\bar{K}_0(A(pt)) \to 1+\sum_{n\geq 1}\; \bar{K}_0(A(pt))\otimes _{\Z} Rep_{\Q}(\Sigma_n) \cdot t^n:\: [\VV]\mapsto 1+\sum_{n\geq 1} \; [\VV^{\otimes n}]\cdot t^n \:,$$
whose projection onto the summand of $t^n$ gives us the $n$-th power map on the level of Grothendieck groups.

\end{proof}

From the inclusion $\Sigma_i\times \Sigma_j\to \Sigma_{i+j}$ one gets homomorphisms
$$Rep_{\Q}(\Sigma_{i+j}) \to Rep_{\Q}(\Sigma_{i}\times \Sigma_j)  \simeq  Rep_{\Q}(\Sigma_i)\otimes Rep_{\Q}(\Sigma_j)$$
and, by duality,
$$ Rep_{\Q}(\Sigma_i)_*\otimes Rep_{\Q}(\Sigma_j)_* \to Rep_{\Q}(\Sigma_{i+j})_*  \:.$$
Therefore,  
$$Rep_{\Q}(\Sigma)_*:= \oplus_{n\geq 0} \; Rep_{\Q}(\Sigma_n)_*$$
becomes a {\em commutative graded ring} (\cite{A}). Denoting by
$$Op\left(\bar{K}_0(A(pt))\right):=map\left( \bar{K}_0(A(pt)) ,\bar{K}_0(A(pt)) \right)$$
the {\em operation ring} of self-maps of $\bar{K}_0(A(pt))$ (with the pointwise addition and multiplication), the group homomorphisms
 $$Rep_{\Q}(\Sigma_i)_*\to Op\left(\bar{K}_0(A(pt))\right):\; \phi\mapsto \phi_{pt}$$
extend additively to a {\em ring homomorphism} 
\begin{equation}\label{R*}
 Rep_{\Q}(\Sigma)_* \to Op\left(\bar{K}_0(A(pt))\right)\:.
\end{equation}

Note that any element $g\in \Sigma_n$ induces a {\em character}
$\phi:=tr(g):   Rep_{\Q}(\Sigma_n)\to \Z$, and therefore an operation on $\bar{K}_0(A(pt))$
by taking the trace of $g$ in the corresponding representation. The character $tr(g)$ depends of course only on the conjugacy class of
$g\in  \Sigma_n$. The most important operations are the following (compare \cite{A}):
\begin{enumerate}
 \item[$\sigma$:] The homomorphisms $\sigma_n:=\frac{1}{n!} \cdot \sum_{g\in \Sigma_n} tr(g): Rep_{\Q}(\Sigma_n)\to \Z$
corresponds to the $n$-th symmetric power operation $[\VV] \mapsto [(\VV^{\otimes n})^{\Sigma_n}]$.
\item[$\lambda$:] The homomorphisms $\lambda_n:=\frac{1}{n!} \cdot \sum_{g\in \Sigma_n} (-1)^{sign(g)}\cdot tr(g): Rep_{\Q}(\Sigma_n)\to \Z$
corresponds to the $n$-th antisymmetric power operation $[\VV] \mapsto [(\VV^{\otimes n})^{sign-\Sigma_n}]$.
\item[$\Psi$:] If $g\in \Sigma_n$ is an $n$-cycle, then $tr(g)$ corresponds by \cite{A}[cor.1.8] to the $n$-th {\em Adams operation}
$\Psi_n: \bar{K}_0(A(pt))\to \bar{K}_0(A(pt))$ associated to the pre-lambda structure $\sigma_t$ (or $\lambda_t$ depending on the chosen conventions).
\end{enumerate}
Finally, $Rep_{\Q}(\Sigma)_*$ is a polynomial ring on the generators $\{\sigma_n|n\in \N\}$ (or $\{\lambda_n|n\in \N\}$),
and $Rep_{\Q}(\Sigma)_*\otimes_{\Z}\Q$ is a polynomial ring on the generators $\{\Psi_n|n\in \N\}$ (\cite{A}).\\

The important point for us is the fact that we have corresponding ``coefficients'' for $\MC\in A(X)$ on the {\em symmetric products}:
\begin{definition}
The group homomorphisms $\phi\in Rep_{\Q}(\Sigma)_*$ discussed above induce the following ``coefficients''  on the {\em symmetric products} $X^{(n)}$:
\begin{enumerate}
\item[$\sigma$:] $\MC^{(n)}=({p_n}_*\MC^{\boxtimes n})^{\Sigma_n} \in A(X^{(n)}) \:.$
\item[$\lambda$:] $\MC^{\{n\}}=({p_n}_*\MC^{\boxtimes n})^{sign-\Sigma_n} \in A(X^{(n)}) \:.$
\item[$\Psi$:] $\Psi_n(\MC):=\phi_{X}(\MC) \in \bar{K}_0(A(X^{(n)}))$, with $\phi:=tr(g)$ for $g\in \Sigma_n$  an $n$-cycle.
\end{enumerate}
\end{definition}
The (anti-)symmetric powers $\MC^{(n)},\MC^{\{n\}}$ have already appeared before.
The decomposition (\ref{K-decomp}) of the Grothendieck group of $\Sigma_n$-equivariant objects is functorial
under maps between spaces with trivial $\Sigma_n$-action, e.g., the constant map $k: X^{(n)}\to pt$. 
This implies the following equality relating our {\em Adams operation} $\Psi_r: ob(A(X))\to  \bar{K}_0(A(X^{(n)}))$
with the  Adams operation of the pre-lambda ring $\bar{K}_0(A(pt))$.
\begin{prop} With the above definitions and notations, we get
\begin{equation} 
 \Psi_r([k_*\MC]) = k_*(\Psi_r(\MC))
\end{equation}
\end{prop}


\section{Appendix: Exterior products and equivariant objects} \label{appA}
In this Appendix we collect in an abstract form some categorical notions needed to formulate a relation between {\em exterior products} and
{\em equivariant objects}, which suffices to show that our Assumptions \ref{ass4} (i)-(iv) are fullfilled in many situations, e.g.,
for the derived categories $D^b_c(X)$ and $D^b_{coh}(X)$. 
 Here we work over a category $space$ of spaces, which for us are the quasi-projective complex algebraic varieties $X$. But the same arguments also apply to other kinds of categories of ``spaces", e.g.,  (compact) topological or  complex analytic spaces. 

 
\subsection{Cofibered Categories}
Let $space$ be a (small) category with finite products $\times$ and terminal object
$pt$ (corresponding to the empty product). The universal property of the product $\times$ makes $space$ into a {\em symmetric monoidal category} \cite{Bo}[sec.6.1], i.e., for $X,Y,Z\in ob(space)$ there are functorial isomorphisms
\begin{align}
 a: (X\times Y)\times Z &\stackrel{\sim}{\to} 
X\times (Y\times Z)\:, \tag{associativity}\\ 
l: pt\times X \stackrel{\sim}{\to} X \quad &\text{and} \quad r: X\times pt \stackrel{\sim}{\to} X \:,\tag{units}\\ 
s: X\times Y &\stackrel{\sim}{\to} Y\times X \:, \tag{symmetry}
\end{align}
such that $s^2=id$ and the following diagrams commute:
\begin{equation} \label{pentagon}\xymatrix{
((W\times X)\times Y) \times Z \ar[d]_{a\times id} \ar[r]^a & (W\times X)\times (Y\times Z) \ar[r]^a & W\times ( X\times (Y\times Z)) \\
(W\times (X\times Y))\times Z \ar[rr]_a & &  W\times ((X\times Y) \times Z) \ar[u]_{id\times a}}
\end{equation}
\begin{equation} \label{unit-ass}\xymatrix{
(X\times pt)\times Y \ar[rr]^a \ar[dr]_{r\times id}& & X\times (pt \times Y)\ar[dl]^{id\times l}\\ 
& X\times Y & } 
\end{equation}
\begin{equation} \label{hexagon}\xymatrix{
(X\times Y) \times Z \ar[d]_{s\times id} \ar[r]^a& X\times (Y\times Z) \ar[r]^s & ( Y\times Z)\times X \ar[d]^a\\
(Y\times X)\times Z \ar[r]_a & Y\times (X \times Z) \ar[r]_{id\times s} & Y\times (Z\times X)}
\end{equation}
and
\begin{equation} \label{unit-sym}\xymatrix{
pt\times X \ar[rr]^s \ar[dr]_l & & X\times pt \ar[dl]^r \\
& X &}
\end{equation}
Let us define inductively $X^0:=pt, X^1:=X$ and $X^{n+1}:=X\times X^n$, so that a morphism
$f: X\to Y$ induces $f^n: X^n\to Y^n$. Then the above constraints for $\times$ imply that
the definition of $X^n$ does not depend (up to canonical isomorphisms) on the chosen order of brackets. Moreover,
$X^n$ gets a canonical induced (left) $\Sigma_n$-action such that $f^n$ is an equivariant morphism.\\

In addition, we fix a covariant pseudofunctor $(-)_{*}$ on $space$ (compare e.g., \cite{Vi}[Ch.3] or \cite{LH}[Part II, Ch.1]), i.e.,
a category $A(X)$ for all $X\in ob(space)$ and (push down) functors $f_*: A(X)\to A(Y)$ for all morphisms $f: X\to Y$ in $space$, together with natural isomorphisms
$$e: id_{A(X)} \stackrel{\sim}{\to} id_{X*}$$
and 
$$c: (gf)_* \stackrel{\sim}{\to} g_*f_*$$
for all composable pairs of morphisms $f,g$ in $space$, such that the following conditions are satisfied:
\begin{enumerate}
\item[(PS1)] For any morphism $f: X\to Y$ the map $$f_*(id_{A(X)})= f_* = (f\circ id_{X})_*\stackrel{c}{\to} f_*(id_{X*})$$ agrees with $f_*(e)$.
\item[(PS2)] For any morphism $f: X\to Y$ the map $$id_{A(Y)}(f_*)= f_* = (id_{Y}\circ f)_*\stackrel{c}{\to} id_{Y*}(f_*)$$ agrees with $e(f_*)$.
\item[(PS3)] For any triple of composable  morphisms 
$$
\begin{CD} X @> f >> Y @> g >> Z @> h >> W \:, \end{CD}$$
the following diagram commutes:
$$
\begin{CD} (hgf)_* @> c >> (hg)_*f_*\\
@V c VV @VV c V \\ 
h_*(gf)_* @>> c > h_*g_*f_* \:. \end{CD}$$ 
\end{enumerate}
Note that in many cases (e.g., in all our applications), $e$  is just the identity $id_{A(X)}=id_{X*}$. The above conditions
allow to make the pairs $(X,M)$ with $X\in ob(space)$ and $M\in A(X)$ into a category
$A/space$, where
a morphism $(f,\phi): (X,M)\to (Y,N)$ is given by a morphism $f: X\to Y$ in $space$ together with a morphism $\phi: f_*(M)\to N$ in $A(Y)$. The composition of morphisms
$$\begin{CD} (X,M) @> (f,\phi) >> (Y,N) @> (g,\psi) >> (Z,O) \end{CD}$$
is defined by
\begin{equation} \begin{CD}
(gf)_*M @> c >> g_*f_*M @> g_*(\phi) >> g_*N @> \psi >> O \:. 
\end{CD}\end{equation}
Then the associativity of the composition follows from (PS3) above, whereas $(id,e^{-1})$ becomes the identity arrow by (PS1) and (PS2). The projection $p: A/space \to space$ onto the first component
defines a functor  making $A/space$ into a cofibered (or sometimes also called opfibered)
category over $space$, with $A(X)$ isomorphic to the fiber over $X$.
Here $(id_X,\tilde{\phi}): (X,M)\to (X,N)$ corresponds to $\tilde{\phi}:={\phi}\circ e:
M\to id_{X*}(M)\to N$. In particular, any morphism $(f,\phi): (X,M)\to (Y,N)$ in $A/space$ can be decomposed as $(f,\phi)= (id_Y, \tilde{\phi})\circ (f,id_{f_*M})$.\\

Sometimes it is more natural to work with the opposite arrows in $A(X)$, and to think of $(-)_*$ as a pseudofunctor with ``values" in the opposite category
$A^{op}(-)$. Then the pairs $(X,M)$ with $X\in ob(space)$ and $M\in A(X)$ become a category
$A^{op}/space$, were
a morphism $(f,\phi): (X,M)\to (Y,N)$ is given by a morphism $f: X\to Y$ in $space$ together with a morphism $\phi: N\to f_*(M)$ in $A(Y)$. The composition of morphisms
$$\begin{CD} (X,M) @> (f,\phi) >> (Y,N) @> (g,\psi) >> (Z,O) \end{CD}$$
is then defined by
\begin{equation} \begin{CD}
O  @> \psi >> g_*N @>g_*(\phi)  >> g_*f_*M @> c^{-1} >> (gf)_*M\:, 
\end{CD}\end{equation}
with $(id,e)$ the identity arrow.


\subsection{Equivariant objects}
Suppose that the  group $G$, with unit $1$, acts (from the left) on $X$, i.e., we have a group homomorphism $\phi: G\to Aut_{space}(X)$. Then a $G$-equivariant object $M\in A(X)$
is by definition given by  a family of isomorphisms 
$$\tilde{\phi}_g: g_*M\to M \quad (g\in G)$$ such that
$$\tilde{\phi}_1=e^{-1} \quad \text{and}  \quad \tilde{\phi}_{gf}= \tilde{\phi}_g\circ g_*(\tilde{\phi}_f) \circ c \quad \text{for all} \quad f,g\in G \:.$$
This just means that $(X,M) \in ob(A/space)$ has a $G$-action given by a group homomorphism $\tilde{\phi}: G\to Aut_{A/space}((X,M))$. With the obvious morphisms, this defines the category $A_G(X)$ of $G$-equivariant objects in $A(X)$. 
If $G$ acts trivially on $X$,
i.e., $g=id_X$ for all $g\in G$, then this corresponds to an action of $G$ on $M$ in $A(X)$, i.e., a group homomorphism $\phi: G\to Aut_{A(X)}(M)$.
For a $G$-equivariant morphism $f: X\to Y$ of $G$-spaces one gets an induced functorial push down $f_*: A_G(X)\to A_G(Y)$, defining a covariant pseudofunctor on the category $G-space$. If we prefer to work with $A^{op}$, we can use the isomorphisms 
$$\tilde{\psi}_g:=\tilde{\phi}_g^{-1}: M\to g_*M \quad (g\in G)$$ such that
$$\tilde{\psi}_1=e \quad \text{and}  \quad \tilde{\psi}_{gf}= c^{-1} \circ g_*(\tilde{\psi}_f) \circ \tilde{\psi}_g\quad \text{for all} \quad f,g\in G \:.$$
Then a $G$-equivariant object $M\in A_G(X)$ corresponds to a $G$-action on   $(X,M)$ in $ob(A^{op}/space)$  given by a group homomorphism $\tilde{\psi}: G\to Aut_{A^{op}/space}((X,M))$.
If $G$ acts trivially on $X$,
then this corresponds to an action of $G$ on $M$ in $A(X)$, i.e., a group homomorphism $\psi: G\to Aut_{A(X)}(M)$ given by isomorphisms
$$\psi_g:=\phi_g^{-1}: M\to M \quad (g\in G) \:$$
And this is the version needed in this paper for $\Sigma_n$-equivariant objects on symmetric products $X^{(n)}$.\\

In our applications, we consider a pseudofunctor $(-)_*$ taking values in a {\em ($\Q$-linear) additive category} $A(-)$,
so that one can look at the corresponding Grothendieck groups $\bar{K}_0(X)$. 
Then $A(X)$ is a ($\Q$-linear) additive category for all
$X\in ob(space)$, with $f_*$ ($\Q$-linear) additive for all morphisms $f$.
 So these induce also homomorphisms of the corresponding 
Grothendieck groups $\bar{K}_0(-)$. If a group $G$ acts on $X$, then the category of $G$-equivariant objects $A_G(X)$ also becomes a 
($\Q$-linear) additive category.\\

Assume now that $A/space$ (or $A^{op}/space$) has the structure of a {\em symmetric monoidal category}, such that
the projection $p: A/space \to space$ (or $p: A^{op}/space \to space$) onto the first component is a strict monoidal functor.
So we have a functorial ``product" (or ``pairing") 
\begin{equation}
(X,M) \boxtimes (Y,N) = (X\times Y, M\boxtimes N), 
\end{equation}
together with associativity, unit and symmetry isomorphisms $a,l,r,s$ in $A/space$ (or $A^{op}/space$) satisfying $s^2=id$, (\ref{pentagon}), (\ref{unit-ass}), (\ref{hexagon}) and (\ref{unit-sym}). This suffices to get the following properties of the Assumptions \ref{ass4}(ii)-(iv) (for the last property compare also with the next Section):
\begin{enumerate}
 \item[(ii')] For $M\in A(X)$ one gets the $\Sigma_n$-equivariant object $M^{\boxtimes n}\in A(X^n)$  corresponding to the $\Sigma_n$-equivariant object $(X,M)^n$ in $A/space$ (or in $A^{op}/space$).
\item[(iii')] If we identify $pt^n$ with $pt$ by the natural projection isomorphism $k: pt^n\stackrel {\sim}{\to} pt$, then $A(pt)$ becomes a symmetric monoidal category
with unit $1_{pt}\in A(pt)$ and product $M\otimes M':=k_*(M\boxtimes M')$.
  \item[(iv')] By the functorality of $\boxtimes$, there is a natural K\"{u}nneth morphism $k_*(M^{\boxtimes n})\to (k_*M)^{\otimes n} \in A_{\Sigma_n}(pt)$
  (or $(k_*M)^{\otimes n}\to  k_*(M^{\boxtimes n})\in A_{\Sigma_n}(pt)$).
\end{enumerate}

If, moreover, the  pseudo-functor $(-)_*$  takes values in a
pseudo-abelian $\Q$-linear additive category $A(-)$, with the tensor product $\otimes$ on $A(pt)$ being 
$\Q$-linear additive in both variables, then all properties of Assumptions \ref{ass4} are fullfilled, if the corresponding K\"{u}nneth morphism (iv') for a constant map is always 
an isomorphism
(as in all our examples, where even  $\boxtimes$ is $\Q$-linear additive in both variables.)


\subsection{K\"{u}nneth morphisms}
If one looks at the morphisms $(f,id_{f_*M}): (X,M)\to (Y,f_*M)$ and $(f',id_{f'_*M'}): (X',M')\to (Y',f'_*M')$ in $A/space$ (or $A^{op}/space$),
then the functoriality of $\boxtimes$ yields by
\begin{equation}\label{Kue-func}
 (f,id_{f_*M})\boxtimes (f',id_{f'_*M'}) =:(f\times f', \Ku)
\end{equation}
 a K\"{u}nneth morphism in $A(Y\times Y')$, namely:
$$\Ku: (f\times f')_*(M\boxtimes M')\to f_*M\boxtimes f'_*M' \quad \text{or} \quad
\Ku: f_*M\boxtimes f'_*M' \to (f\times f')_*(M\boxtimes M') \:,$$
which is functorial in $M\in ob(A(X))$ and $M'\in ob(A(X'))$.  In many examples, $\Ku$ is an isomorphism, and one can easily switch between the two viewpoints.
In all our applications, one has a natural K\"{u}nneth morphism
\begin{equation}\label{Kue}
 \Ku: f_*M\boxtimes f'_*M' \to (f\times f')_*(M\boxtimes M'),
\end{equation}
so that we have to work with a pseudofunctor with ``values" in the opposite category
$A^{op}(-)$. We only spell out here where such a {\em symmetric monoidal structure} on $A^{op}/space$ (over $(space,\times)$) really comes from in this case. First of all, the ``pairing" (or bifunctor)
$$\boxtimes: A^{op}/space \times A^{op}/space \to A^{op}/space$$
corresponds to 
\begin{itemize}
\item[(p1)] An exterior product $\boxtimes: A(X)\times A(X') \to A(X\times X'); \: (M,M')\mapsto M\boxtimes M'$, bifunctorial in $M,M'$.
\item[(p2)] A covariant pseudofunctor $(-)_*$ with ``values" in the opposite category
$A^{op}(-)$.
\item[(p3)] For all morphisms $f: X\to Y$ and $f': X'\to Y'$ in $space$, a  K\"{u}nneth morphism in $A(Y\times Y')$:
$$\Ku: f_*M\boxtimes f'_*M' \to (f\times f')_*(M\boxtimes M')\:,$$ 
functorial in $M\in ob(A(X))$ and $M'\in ob(A(X'))$. 
\item[(p4)]  For $f,f',M,M'$ as in (p3), and for morphisms $g: Y\to Z,\: g': Y'\to Z'$ in $space$, these satisfy the compability
\end{itemize}
$$\xymatrix{
 (gf)_*M\boxtimes (g'f')_*M' \ar[d]_{c\boxtimes c} \ar[rr]^{\Small{\Ku}}  &  & (gf\times g'f')_*(M\boxtimes M') \ar[d]^{c}\\ 
g_*f_*M \boxtimes g'_*f'_*M' \ar[r]_-{\Small{\Ku}} & (g\times g')_*(f_*M\boxtimes f'_*M') \ar[r]_-{\Small{\Ku}} &
(g\times g')_*(f\times f')_*(M\boxtimes M')\:.} $$

Then the associativity, unit and symmetry isomorphisms $a,l,r,s$ in $A^{op}/space$ are given by isomorphisms
functorial in $M\in A(X), M'\in A(Y)$ and $M''\in A(Z)$:
\begin{align}
M\boxtimes (M' \boxtimes M'')  &\stackrel{\sim}{\to} 
a_*((M\boxtimes M')\boxtimes M'')\:,\\ 
M \stackrel{\sim}{\to} l_*(1_{pt}\boxtimes M) \quad &\text{and} \quad M \stackrel{\sim}{\to} r_*(M\boxtimes 1_{pt}) \:,\\ 
M\boxtimes M' &\stackrel{\sim}{\to} s_*(M'\boxtimes M) \:. 
\end{align}
And these have to be compatible with the  K\"{u}nneth morphism (\ref{Kue}) according to the functoriality of $\boxtimes$ given by (\ref{Kue-func}).
Note that for the definition of $\boxtimes$, we only have to work over the symmetric monoidal subcategory $space^{iso}$ of $space$,
with the same objects and only all isomorphisms as morphisms. In this case one can easily switch between covariant and contravariant pseudofunctors
using $f^*:=(f^{-1})_*$. 


\subsection{Examples}
In most applications, the exterior product $\boxtimes$ comes from an interior product $\otimes$ on $A(X)$ making it into a  symmetric monoidal 
category with unit $1_X\in A(X)$, together with a contravariant pseudofunctor $f^*$ on $space$ compatible with the symmetric monoidal structure (e.g., $f^*1_Y\simeq 1_X$ 
and $f^*(- \times -) \simeq f^*(-) \times f^*(-)$ for a morphism $f: X\to Y$ in $space$). Indeed, using the projections 
$$\begin{CD} X @< p << X\times X' @> q >> X' \end{CD}$$
one defines for $M\in A(X)$ and $M'\in A(X')$ the exterior product by $$M\boxtimes M':= p^*M\otimes q^*M'\:.$$ Then one only needs in addition a K\"{u}nneth morphism (p3) for the
covariant pseudofunctor $(-)_*$ satisfying the compability (p4). Here are different examples showing how to get such a structure:
 \begin{example}[adjoint pair]  The functors $f_*$ are right adjoint to $f^*$ so that the pseudo\-functors $(-)^*,(-)_*$ form an adjoint pair as in \cite{LH}[Part I, sec.3.6].
Here the K\"{u}nneth morphism $\Ku$ for the cartesian diagram
\begin{equation}
\label{adj}
\begin{CD} X @< p << X\times X' @> q >> X' \\
   @V f VV @VV f\times f' V @VV f' V\\
Y @<< p' < Y\times Y' @>> q' > Y' \:,
  \end{CD}
\end{equation}
with $M\in A(X)$ and $M'\in A(X')$ is induced by adjunction from the morphism
\begin{align*}
 (f\times f')^*(f_*M\boxtimes f'_*M') & = (f\times f')^*(p'^*f_*M\otimes q'^*f'_*M') \\
& \simeq ((f\times f')^*p'^*f_*M) \otimes ((f\times f')^*q'^*f'_*M')\\
&\simeq (p^*f^*f_*M) \otimes (q^*f'^*f_*M) \\
&\stackrel{adj}{\longrightarrow} p^*M\otimes q^*M' = M\boxtimes M' \:.
\end{align*}
\end{example}
A typical example is given by $A(X):=D(X)$, the derived category of sheaves of $\OO_X$-modules, for $X$ a commutative ringed space, with $f^*=Lf^*$ and $f_*=Rf_*$ the derived inverse and direct image (\cite{LH}[Part I, (3.6.10) on p.124]). Or we can work with suitable subcategories, e.g., the subcategory $D_{qc}(X)$ of complexes with quasi-coherent cohomology
in the context of $X$ a separated scheme. Here are some important examples of such ``adjoint pairs":
\begin{itemize}
 \item[(coh*)] We work with the category $space$ of separated schemes of finite type over a base field $k$, with $A'(X):=D_{qc}(X)$ and $f^*:=Lf^*, f_*:=Rf_*$ as before, see \cite{LH}.
Note that the subcategory $A(X):=D^b_{coh}(X)$ of bounded complexes with coherent cohomology is in general not stable under $f^*$ or $\otimes$.
But it is stable under $f_*$ for proper morphisms, and under the exterior products $\boxtimes$, defining therefore a symmetric monoidal structure on $(D^b_{coh})^{op}/space$
with $space:=sch_k^{cp}$ the category of separated complete schemes of finite type over a base field $k$.
\item[(c*)] We work with the derived category $A(X):=D^b_c(X)$ of bounded complexes (of sheaves of vector spaces)  with constructible cohomology in the complex algebraic (or analytic)
context, with $space$ the category of complex (quasi-projective) varieties (or the category of compact complex analytic spaces) \cite{Sch}. Here we  use $\otimes$ and $f^*:=Lf^*, f_*:=Rf_*$ to get a symmetric monoidal structure on $(D^b_c)^{op}/space$.\\
The same arguments also work for constructible sheaf complexes in the context of real geometry for semialgebraic or compact subanalytic sets, and for stratified maps between suitable compact stratified sets \cite{Sch}.
\end{itemize}  
Also note that in all these  examples the K\"{u}nneth morphism $\Ku$ is an isomorphism
(compare, e.g., with \cite{Bon}[Thm.2.1.2] for the case (coh*), using the fact that a projection $X\times X'\to X$ is a flat morphism,
and \cite{Sch}[Sec.1.4, Cor.2.0.4] or \cite{MSS}[Sec.3.8] for the case (c*)). 

\begin{example}[base change $+$ projection morphism] Assume that  the pseudofunctor $f_*:=f_!$ is endowed with a natural {\em projection morphism} 
\begin{equation}
(f_!M)\otimes M' \to  f_!(M\otimes f^*M')
\end{equation}
and a natural   {\em base change morphism}
\begin{equation}
 p'^*f_! M\to (f\times id_{X'})_!p^*M  
\end{equation}
for $f,p,p'$ as in the cartesion diagram (\ref{adj}) (with $X'=Y'$), satisfying suitable compatibilities. 
Then one gets the K\"{u}nneth morphism $\Ku$ for $f\times id_{X'}$ by
\begin{align*}
 (f_!M)\boxtimes M' &= (p'^*f_!M)\otimes q'^*M'\\
& \to ((f\times id_{X'})_!p^*M)\otimes   q'^*M'\\
&\to (f\times id_{X'})_!( p^*M\otimes ((f\times id_{X'})^*q'^*M')\\
&\simeq (f\times id_{X'})_!( p^*M\otimes q^*M') \\
&= (f\times id_{X'})_!(M\boxtimes M') \:.
\end{align*}
In the same way, one gets the K\"{u}nneth morphism for $id_X\times f'$ by using the symmetry isomorphism $s$, and the K\"{u}nneth morphism for
$(f\times f')= (f\times id_{X'})\circ (id_X\times f')$ follows by composition.
\end{example}

Here is an important example of such a ``pair" $(-)^*,(-)_!$ with projection and base change morphisms:
\begin{itemize}
\item[(c!)] We work with the derived category $A(X):=D^b_c(X)$ of bounded complexes (of sheaves of vector spaces)  with constructible cohomology in the complex algebraic 
context, with $space$ the category of complex (quasi-projective) varieties \cite{Sch}. Here we use $\otimes, f^*:=Lf^*$ and  the derived direct image functor with proper support $f_!:=Rf_!$ to get in this way a symmetric monoidal structure on $(D^b_c)^{op}/space$.\\
The same arguments also work for constructible sheaf complexes in the context of real geometry for semialgebraic sets \cite{Sch}.
\end{itemize}  
Also in this example, the projection and base change morphisms, and therefore also the 
K\"{u}nneth morphisms $\Ku$, are isomorphisms (compare, e.g., with
\cite{Sch}[Sec.1.4]).\\

Finally, note that there are also many interesting cases where one doesn't have such 
a ``tensor structure"
on $A(X)$ for a singular space $X$, e.g., for perverse sheaves or coherent sheaves. But nevertheless $A^{op}/space$ can be endowed with a symmetric monoidal structure as above so that our techniques and results apply.

\begin{ack} The authors thank Morihiko Saito, Sylvain Cappell and Shoji Yokura for discussions on the subject of this paper.
The authors acknowledge the support of New York University, where this work was started. The first author also acknowledges support from the Max-Planck-Institut f\"ur Mathematik, Bonn, where part of this project was written.
\end{ack}


\begin{thebibliography}{00}
\bibitem{A} M. F. Atiyah, \emph{Power operations in K-theory},
Quart. J. Math. {\bf 17} (1966), 165--193.


\bibitem{BS} P. Balmer, M. Schlichting, \emph{Idempotent completion of triangulated categories}, J. Algebra {\bf 236} (2001), no. 2, 819--834.


\bibitem{BFM} P. Baum, W. Fulton, R. MacPherson,
\emph{Riemann-Roch for singular varieties},  Inst. Hautes \'Etudes Sci. Publ. Math.  {\bf 45} (1975), 101--145.

\bibitem{BL} J. Bernstein, V. Lunts, \emph{Equivariant sheaves and functors}, Lecture Notes in Mathematics, {\bf 1578}, Springer-Verlag, Berlin, 1994.

\bibitem{Bi} S. Biglari, \emph{A K\"unneth formula in tensor triangulated categories},
J. Pure Appl. Algebra {\bf 210} (2007), no. 3, 645--650. 


\bibitem{Bon} J. Bonsdorff, \emph{A Fourier transformation for Higgs bundles},
J. Reine Angew. Math. {\bf 591} (2006), 21--48.

\bibitem{Bo} F. Borceux, \emph{Handbook of Categorical Algebra 2},
Encyclopedia of Mathematics and its Applications 51, Combridge University Press, 1994.

\bibitem{BL3} L. Borisov, A. Libgober,  \emph{Elliptic genera of singular varieties, orbifold elliptic genus and chiral de Rham complex}, in Mirror symmetry, IV (Montreal, QC, 2000), 325-342, AMS/IP Stud. Adv. Math., 33, Amer. Math. Soc., Providence, RI, 2002.


\bibitem{BSY} J. P. Brasselet, J. Sch\"urmann, S. Yokura, \emph{Hirzebruch classes and motivic Chern classes of singular
spaces}, J. Topol. Anal. {\bf 2} (2010), no. 1, 1--55.

\bibitem{eCMS} S. E. Cappell,  L. Maxim, J. L. Shaneson, \emph{Equivariant genera of complex algebraic varieties}, Int. Math. Res. Notices {\bf 2009} (2009), 2013-2037.

\bibitem{CMSS} S. E. Cappell,  L. Maxim, J. Sch\"urmann, J. L. Shaneson, \emph{Equivariant Hirzebruch classes of complex algebraic varieties}, 
arXiv:1004.1844.

\bibitem{CMSSY} S. E. Cappell,  L. Maxim, J. Sch\"urmann, J. L. Shaneson, S. Yokura, \emph{Characteristic classes of symmetric products of complex quasi-projective varieties},
arXiv:1008.4299.

\bibitem{Che} J. Cheah, \emph{On the cohomology of Hilbert schemes of points},
J. Algebraic Geometry {\bf 5} (1996), no. 3, 479-511.

\bibitem{De} P. Deligne, \emph{Th\'{e}orie de Hodge, II, III},
Inst. Hautes \'Etudes Sci. Publ. Math. {\bf 40, 44} (1972, 1974).

\bibitem{De2} P. Deligne, \emph{Cat\'{e}gories tensorielles},
Mosc. Math. J. {\bf 2} (2002), 227--248.

\bibitem{Ga} N. Ganter, \emph{Orbifold genera, product formulas and power operations},  Adv. Math. {\bf 205} (2006), no. 1, 84--133.

\bibitem{Ge1} E. Getzler, \emph{Mixed Hodge structures on configuration spaces},
arXiv:math/9510018.

\bibitem{Ge2} E. Getzler, \emph{Resolving mixed Hodge modules on configuration spaces},
Duke Math. J. {\bf 96} (1999), no. 1, 175--203.



\bibitem{GM1} M. Goresky, R. MacPherson, \emph{Intersection Homology}, Topology {\bf 19} (1980), 135--162.

\bibitem{Go} E. Gorsky, \emph{Adams operations and power structures},
Mosc. Math. J. {\bf 9} (2009), no. 2, 305--323.

\bibitem{GS} L. G\"ottsche, W. Soergel,  \emph{Perverse sheaves and the cohomology of Hilbert schemes of smooth algebraic surfaces}, Math. Ann. {\bf 296} (1993), 235-245.

\bibitem{G} A. Grothendieck, \emph{ Sur quelques points d'alg\`ebre homologique.} T\^ohoku Math. J. (2) {\bf 9} 1957, 119--221.

\bibitem{GLM0} S. M. Gusein-Zade, I. Luengo, A. Melle-Hern\'andez, \emph{A power structure over the Grothendieck ring of varieties.}  Math. Res. Lett.  {\bf 11}  (2004),  no. 1, 49--57. 

\bibitem{GLM1} S. M. Gusein-Zade, I. Luengo, A. Melle-Hern\'andez, \emph{Power structure over the Grothendieck ring of varieties and generating series of Hilbert schemes of points.}  Michigan Math. J.  {\bf 54}  (2006),  no. 2, 353--359.



\bibitem{HRV} T. Hausel, F. Rodriguez-Villegas, \emph{Mixed Hodge polynomials of character varieties},
Inv. Math. {\bf 174} (2008), 555--624.

\bibitem{Hl} F. Heinloth, \emph{A note on functional equations for zeta functions with values in Chow motives},
Annales de L'Inst. Fourier {\bf 57} (2007), 1927--1945.

\bibitem{Ka} M. Kapranov, \emph{The elliptic curve in S-duality theory and Eisnestein series for Kac-Moody groups},
arXiv:math/0001005.

\bibitem{Kn} D. Knutson, \emph{$\lambda$-Rings and the Representation Theory of the Symmetric Group},
Lecture Notes in Mathematics, Vol. {\bf 308}. Springer-Verlag, Berlin-New York, 1973. 

\bibitem{LC} J. Le, X.-W. Chen, \emph{Karoubianness of a triangulated category},
  Journal of Algebra  {\bf 310} (2007), 452--457.

\bibitem{LH} J. Lipmann, M. Hashimoto, \emph{Foundations of Grothendieck Duality for Diagrams of Schemes},
Lecture Notes in Mathematics, Vol. {\bf 1960}. Springer-Verlag, Berlin-New York, 2009. 


\bibitem{Mac1} I. G. Macdonald, \emph{Symmetric products of an algebraic curve},  Topology  {\bf 1}  (1962), 319-343.

\bibitem{Mac}  I. G. Macdonald, \emph{The Poincar\'e polynomial of a symmetric product}, Proc. Cambridge Philos. Soc. {\bf 58}, 1962, 563 - 568.

\bibitem{MP} R. MacPherson, \emph{Chern classes for singular algebraic varieties},
Ann. of Math. (2)  {\bf 100}  (1974), 423--432.

\bibitem{MSS} L. Maxim, J. Sch\"{u}rmann, M. Saito, \emph{Symmetric products of mixed Hodge modules}, arXiv:1008.5345.





\bibitem{M} B. Moonen, \emph{Das Lefschetz-Riemann-Roch-Theorem f\"ur singul\"are Variet\"aten}, Bonner Mathematische Schriften {\bf 106} (1978),  viii+223 pp.

\bibitem{Oh} T. Ohmoto, \emph{Generating functions for generating Chern classes I: symmetric products},  Math. Proc. Cambridge Philos. Soc.  {\bf 144}  (2008),  no. 2, 423-438. 


\bibitem{Sa} M. Saito, \emph{Mixed Hodge Modules},
Publ. Res. Inst. Math. Sci. {\bf 26}  (1990),  no. 2, 221--333.

\bibitem{Sa0} M. Saito, \emph{Introduction to mixed Hodge modules}, Actes du Colloque de Th\'eorie de Hodge (Luminy, 1987),  Ast\'erisque  No. {\bf 179-180}  (1989), 10, 145-162. 

\bibitem{Sa1} M. Saito, \emph{Modules de Hodge polarisables}, Publ. Res. Inst. Math. Sci. {\bf 24} (1988), no. 6, 849-995.


\bibitem{Sa3} M. Saito, \emph{Mixed Hodge complexes on algebraic varieties}, Math. Ann. {\bf 316} (2000), 283--331.


\bibitem{Sch} J. Sch\"{u}rmann, \emph{Topology of singular spaces and constructible sheaves},
Monografie Matematyczne 63 (New Series), Birkh\"{a}user, Basel, 2003.


\bibitem{SGA4} \emph{S\'{e}m. g\'{e}om\'{e}trie alg\'{e}brique (1967-1969). Th\'{e}orie des topos et cohomologie \'{e}tale. Part III},
Lecture Notes in Mathematics, Vol. {\bf 305}. Springer-Verlag, Berlin-New York, 1973. 

\bibitem{Vi} A. Vistoli, \emph{Grothendieck topologies, fibered categories and decent theory},
Math. Surveys Monogr. {\bf 123} (2005), 1--104.

\bibitem{Za}  D. Zagier, \emph{Equivariant Pontrjagin classes and applications to orbit spaces. Applications of the $G$-signature theorem to transformation groups, symmetric products and number theory}, Lecture Notes in Mathematics, Vol. {\bf 290}. Springer-Verlag, Berlin-New York, 1972. 

\bibitem{Zh} J. Zhou, \emph{Calculations of the Hirzebruch $\chi_y$ genera of symmetric products by the holomorphic Lefschetz formula}, arXiv:math/9910029.

\end{thebibliography}
\end{document}